\documentclass{amsart}
\usepackage[T1]{fontenc}

\usepackage{hyperref}
\usepackage{amscd,amssymb}
\usepackage[all]{xy}

\allowdisplaybreaks[2]

\title[NS correspondence for Real and Quaternionic vector bundles]{On the Narasimhan-Seshadri correspondence for Real and Quaternionic vector bundles}
\author{Florent Schaffhauser}
\address{Departamento de Matem\'aticas, Universidad de Los Andes, Bogot\'a, Colombia.}
\email{florent@uniandes.edu.co}
\thanks{The author was supported in part by the European Research Council, under the {\em European Community}'s seventh Framework Programme (FP7/2007-2013)/ERC {\em grant agreement} No.FP7-246918, during the course of this work.}

\newcommand{\fu}{\mathfrak{u}}

\newcommand{\C}{\mathbb{C}}
\newcommand{\R}{\mathbb{R}}
\newcommand{\Z}{\mathbb{Z}}
\renewcommand{\H}{\mathbb{H}}

\newcommand{\rk}{\mathrm{rk}}
\renewcommand{\mod}{\mathrm{mod}\,}
\newcommand{\Fix}{\mathrm{Fix}}
\newcommand{\Gal}{\mathrm{Gal}}
\newcommand{\Br}{\mathrm{Br}}
\newcommand{\Id}{\mathrm{Id}}
\newcommand{\Aut}{\mathrm{Aut}}
\newcommand{\Hom}{\mathrm{Hom}}

\newcommand{\gr}{\mathrm{gr}}
\newcommand{\vol}{\mathrm{vol}}

\newcommand{\Out}{\mathrm{Out}}
\newcommand{\Int}{\mathrm{Int}}
\newcommand{\End}{\mathrm{End}}
\newcommand{\Hol}{\mathrm{Hol}}

\newcommand{\cE}{\mathcal{E}}
\newcommand{\cF}{\mathcal{F}}
\newcommand{\cC}{\mathcal{C}}
\newcommand{\cEs}{\sigma(\mathcal{E})}
\newcommand{\cG}{\mathcal{G}}
\newcommand{\cA}{\mathcal{A}}
\newcommand{\cL}{\mathcal{L}}
\newcommand{\cB}{\mathcal{T}}
\newcommand{\cM}{\mathcal{M}}
\newcommand{\cR}{\mathcal{R}}
\newcommand{\cT}{\mathcal{T}}
\newcommand{\cZ}{\mathcal{Z}}
\newcommand{\cW}{\mathcal{W}}

\newcommand{\U}{\mathbf{U}}
\newcommand{\PU}{\mathbf{PU}}
\renewcommand{\O}{\mathbf{O}}
\newcommand{\Sp}{\mathbf{Sp}}
\newcommand{\GL}{\mathbf{GL}}

\renewcommand{\phi}{\varphi}
\newcommand{\si}{\sigma}
\newcommand{\sit}{\widetilde{\sigma}}
\newcommand{\lra}{\longrightarrow}
\newcommand{\lmt}{\longmapsto}

\newcommand{\Msi}{M_\Sigma}
\newcommand{\piRx}{\pi_1(\Msi;x)}
\newcommand{\piCx}{\pi_1(M;x)}
\newcommand{\piR}{\pi_1(\Msi)}
\newcommand{\piC}{\pi_1(M)}
\newcommand{\Si}{\Sigma}
\newcommand{\piL}{\pi_1(S(L))}
\newcommand{\piLR}{\pi_1(S(L)_\Si)}
\newcommand{\Mt}{\widetilde{M}}

\newcommand{\Om}{\Omega}

\newcommand{\RP}{\R\mathbf{P}}
\newcommand{\usit}{u_{\sit}}
\newcommand{\tausit}{\widetilde{\tau}}
\newcommand{\tausi}{\tau}
\newcommand{\eps}{\varepsilon}
\newcommand{\tauL}{\tau_L}
\newcommand{\etat}{\widetilde{\eta}}
\newcommand{\vw}{\vec{w}}
\newcommand{\ModRd}{\mathcal{M}_{\R}(r,d)}
\newcommand{\ModRdw}{\mathcal{M}_{\R}(r,d,\vw)}
\newcommand{\ModRdws}{\mathcal{N}_{\R}(r,d,\vw)}
\newcommand{\ModRds}{\mathcal{N}_{\R}(r,d)}
\newcommand{\ModHd}{\mathcal{M}_{\H}(r,d)}
\newcommand{\ModCd}{\mathcal{M}_{\C}(r,d)}
\newcommand{\ModCds}{\mathcal{N}_{\C}(r,d)}
\newcommand{\ModHds}{\mathcal{N}_{\H}(r,d)}
\newcommand{\Holx}{\mathrm{Hol}_x(A)}
\newcommand{\HolxSi}{\mathrm{Hol}_x^{\Sigma}(A)}
\newcommand{\at}{\widetilde{\alpha}}
\newcommand{\fibre}{F^{-1}\left(\{i\frac{d}{r}\ \vol_M\ \Id_E\}\right)}
\newcommand{\ga}{\gamma}
\newcommand{\Ga}{\Gamma}
\newcommand{\la}{\lambda}
\renewcommand{\rho}{\varrho}
\newcommand{\phis}{\varphi_\si}
\newcommand{\quot}{/\negmedspace/}

\newcommand{\ov}[1]{\overline{#1}}
\newcommand{\os}[1]{\ov{(\si^{-1})^*#1}}

\newtheorem{proposition}{Proposition}[section]
\newtheorem{theorem}[proposition]{Theorem}
\newtheorem{corollary}[proposition]{Corollary}
\newtheorem{lemma}[proposition]{Lemma}

\theoremstyle{definition}
\newtheorem{definition}[proposition]{Definition}
\newtheorem*{ack}{Acknowledgments}
\newtheorem{remark}[proposition]{Remark}


\date{\today}

\numberwithin{equation}{section}

\begin{document}

\begin{abstract}
Let $(M,\si)$ be a compact Klein surface of genus $g\geq 2$ and let $E$ be a smooth Hermitian vector bundle on $M$. Let $\tau$ be a Real or Quaternionic structure on $E$ and denote respectively  by $\cG_\C^\tau$ and $\cG_E^{\,\tau}$ the groups of complex linear and unitary automorphisms of $E$ that commute to $\tau$. In this paper, we study the action of $\cG_\C^\tau$ on the space $\cA_E^{\,\tau}$ of $\tau$-compatible unitary connections on $E$ and show that the closure of a semi-stable $\cG_\C^\tau$-orbit contains a unique $\cG_E^{\,\tau}$-orbit of projectively flat connections. We then use this invariant-theoretic perspective to prove a version of the Narasimhan-Seshadri correspondence in this context: $S$-equivalence classes of semi-stable Real and Quaternionic vector bundes are in bijective correspondence with equivalence classes of certain appropriate representations of orbifold fundamental groups of Real Seifert manifolds over the Klein surface $(M,\si)$.
\end{abstract}

\maketitle

\tableofcontents

\section{Introduction}

\subsection{Background on Klein surfaces}

In this paper, a Klein surface (or Real Riemann surface) is a pair $(M,\si)$ where $M$ is a Riemann surface and $\si:M\lra M$ is an anti-holomorphic involution (or Real structure). A homomorphism $f:(M_1,\si_1) \lra (M_2,\si_2)$ between two Klein surfaces is a holomorphic map $f:M_1\lra M_2$ such that $f\circ \si_1 = \si_2\circ f$. We shall always assume that $M$ is connected. It is convenient to think of $\si$ as the non-trivial element of the group $\Si:=\Gal(\C/\R)\simeq \Z/2\Z$, whose trivial element will be denoted by $1_\Si$ or sometimes simply by $1$. In particular, the group $\Si$ acts on $M$ and we will denote indifferently by $M/\si$ or $M/\Si$ the orbit space of that action, endowed with the quotient topology. Likewise, the fixed-point set of the $\Si$-action on $M$ will be denoted by $M^\si$ or $M^\Si$. The topological classification of compact Klein surfaces (up to $\Si$-equivariant homeomorphism) is well-known and depends upon three numbers:
\begin{itemize}
\item the genus $g$ of $M$,
\item the number $n\in\{0;\ldots;g+1\}$ of connected components of $M^\si$,
\item the number $a\in\{0;1\}$ which is non-zero if and only if the surface $M/\si$ is non-orientable.
\end{itemize} The inequality $n\leq g+1$ is called Harnack's inequality and the invariants $(g,n,a)$ are subject to the following conditions:
\begin{itemize}
\item $n=0 \Rightarrow a=1$,
\item $n=(g+1) \Rightarrow a=0$,
\item $a=0\Rightarrow n \equiv (g+1) \mod 2$.
\end{itemize}

\noindent The triple $(g,n,a)$ will be called the topological type of the Klein surface $(M,\si)$.

\subsection{Real and Quaternionic vector bundles}\label{background_vb}

Atiyah initiated the study of Real vector bundles over general Real spaces in \cite{Atiyah_reality}. Dupont then introduced Symplectic vector bundles in \cite{Dupont}, which nowadays are more commonly called Quaternionic vector bundles (\cite{Hartshorne_CMP}). Over a Klein suface, the definition goes as follows.

\begin{definition}\label{def_real_and_quat_bundles}
Let $(M,\si)$ be a Klein surface. A Real vector bundle over $(M,\si)$ is a pair $(\cE,\tau)$ where $\cE\lra M$ is a holomorphic vector bundle and $\tau:\cE\lra \cE$ is an anti-holomorphic map such that:
\begin{enumerate}
\item The following diagram commutes: $$\begin{CD} \cE @>{\tau}>> \cE \\
@VVV @VVV \\
M @>{\si}>> M.
\end{CD}$$
\item $\forall v\in\cE$, $\forall \la\in\C$, $\tau(\la v) = \ov{\la}\tau(v)$.
\item $\tau^2=\Id_{\cE}$.
\end{enumerate}

A homomorphism between two Real vector bundles $(\cE_1,\tau_1)$ and $(\cE_2,\tau_2)$ is a homomorphism of holomorphic vector bundles $\phi:\cE_1\lra\cE_2$ satisfying the additional condition $\phi\circ\tau_1=\tau_2\circ\phi$.

A Quaternionic vector bundle over $(M,\si)$ is a pair $(\cE,\tau)$ as above but satisfying $\tau^2=-\Id_{\cE}$ instead of $\tau^2=\Id_{\cE}$. Homomorphisms of Quaternionic vector bundles are defined as in the Real case.
\end{definition}

The stability condition for Real vector bundles was introduced in \cite{Wang_NS} then extended to the Quaternionic case in \cite{Sch_JSG}. Moduli spaces of semi-stable Real and Quaternionic vector bundles were constructed using gauge theory in \cite{BHH} and \cite{Sch_JSG}, where they were related to the Real points of moduli spaces of semi-stable holomorphic vector bundles, but the first instance of such a construction actually goes back to Wang (\cite{Wang_moduli}). In the course of the paper, we will use the following notation: $\ModCd$ is the space of $S$-equivalence classes of semi-stable holomorphic vector bundles of rank $r$ and degree $d$ (\cite{Seshadri}) and we denote by $\ModRd$ (resp.\ $\ModHd$) the space of $S$-equivalence classes of semi-stable Real (resp.\ Quaternionic) vector bundles of rank $r$ and degree $d$ (see Definition \ref{def_moduli_spaces_Real_Quat_bdles}). Likewise, we denote by $\ModCds$ the space of isomorphism classes of stable holomorphic vector bundles of rank $r$ and degree $d$ (\cite{Mumford_ICM}) and by $\ModRds$ (resp.\ $\ModHds$) the space of isomorphism classes of geometrically stable Real (resp.\ Quaternionic) vector bundles of rank $r$ and degree $d$ (see Definition \ref{def_stability_cond} for the notion of geometric stability). There is an action $\cE\lmt \ov{(\si^{-1})^*\cE}$ of $\Si=\Gal(\C/\R)$ on $\ModCd$, preserving the subset $\ModCds$ and such that the fixed point-set of the induced action is 
\begin{equation}\label{Glois_inv_geom_stable_bdles}
\ModCds^\Si \ \simeq \ \ModRds \ \sqcup \ \ModHds
\end{equation} (see Propositions \ref{self_conj_simple_bundle} and \ref{moduli_of_geom_stable_Real_and_quat}). We point out that, depending on the topological type $(g,n,a)$ of the Klein surface $(M,\si)$, as well as of the particular values of $r$ and $d$, either $\ModRds$ or $\ModHds$, or even both of them, might be empty: the proof of this is given in \cite{Sch_JSG} and uses the following topological classification result of Biswas, Huisman and Hurtubise in \cite{BHH} (see also \cite{KW} for Part (1)). We formulate it for real and Quaternionic Hermitian vector bundles $(E,\tau)$, where $\tau$ is now an isometry of the smooth Hermitian vector bundle $E$, otherwise satisfying the same three conditions as in Definition \ref{def_real_and_quat_bundles}. Note that, if $M^\si\neq\emptyset$ and $(E,\tau)$ is Real, then $E^\tau\lra M^\si$ is a real vector bundle in the ordinary sense (with fiber $\R^r$) over a disjoint union of $n$ circles, so $E^\tau$ is determined up to isomorphism by its first Stiefel-Whitney class $\vw=w_1(E^\tau) \in H^1(M^\si;\Z/2\Z)\simeq (\Z/2\Z)^n.$ If $n>0$, we set $\vw=(s_1,\ldots,s_n)\in (\Z/2\Z)^n$ and $|\vw| = s_1+\ldots+s_n$ and, if $n=0$, we set $\vw=|\vw|=0$. Also, we denote by $r$ the rank of a complex vector bundle $E$ over $M$ and by $d$ its degree.

\begin{theorem}[\cite{BHH}]\label{top_classif}
Let $(M,\si)$ be a compact Klein surface of topological type $(g,n,a)$.
\begin{enumerate}
\item Two Real Hermitian vector bundles $(E,\tau)$ and $(E',\tau')$ are isomorphic if and only if $(r,d,\vw)=(r',d',\vw')$. The triple $(r,d,\vw)$ is called the topological type of $(E,\tau)$. A necessary and sufficient condition for a Real Hermitian vector bundle of topological type $(r,d,\vw)$ to exist is that $|\vw|=d\,\mod\,2$. In particular, if $M^\si=\emptyset$, the degree of a Real vector bundle must be even and we simply write $(r,d)$ for $(r,d,0)$. If $M^\si\neq\emptyset$, there are $2^{n-1}$ possible topological types of Real vector bundles of rank $r$ and degree $d$.
\item Two Quaternionic Hermitian vector bundles $(E,\tau)$ and $(E',\tau')$ are isomorphic if and only if $(r,d)=(r',d')$. The pair $(r,d)$ is called the topological type of $(E,\tau)$. A necessary and sufficient condition for a Quaternionic Hermitian vector bundle of topological type $(r,d)$ to exist is that $d+r(g-1)\equiv 0\ (\mod\, 2)$. If $M^\si\neq\emptyset$, the rank of a Quaternionic vector bundle must be even so the previous condition implies that its degree, too, must be even.
\end{enumerate} 
\end{theorem}

\noindent In what follows, it will be convenient to write $(r,d,\vw)$ for the topological type of a Real or Quaternionic vector bundle $(E,\tau)$ and to interpret $\vw$ as $0$ and $(r,d,\vw)$ as $(r,d)$ when $M^\si=\emptyset$ or $\tau^2=-\Id_E$.

A first consequence of Theorem \ref{top_classif} is that, in the Real case, it is relevant to introduce moduli spaces $\ModRdw$ (resp.\ $\ModRdws$), consisting of $S$-equivalence (resp.\ isomorphism) classes of semi-stable (resp.\ geometrically stable) Real vector bundles of topological type $(r,d,\vw)$. In particular, one has the following refinement of \eqref{Glois_inv_geom_stable_bdles}: $$\quad\ModCds^\Si \quad\simeq\quad \Big(\bigsqcup_{\vw\, |\, |\vw|=d\,\mod\,2} \ModRdws\quad \Big) \quad \bigsqcup \quad \ModHds\ ,$$ the point being that now each of the moduli spaces $\ModRdws$ and $\ModHds$ is connected (see \cite{Sch_JSG} for a proof and a precise count of the connected components of $\ModCds^\Si$ according to the possible values of $(g,n,a)$ and $(r,d)$). For each topological type of Real or Quaternionic vector bundle, one can fix a Real or Quaternionic Hermitian vector bundle $(E,\tau)$ of that type and consider, instead of $\ModRdw$ or $\ModHd$, the space $\cM^{ss}(E,\tau)$ of $S$-equivalence classes of $\tau$-compatible semi-stable holomorphic structures on it. This point of view is recalled in detail in Section \ref{moduli_of_real_and_quat_bundles} (see in particular the gauge-theorectic construction of $\cM^{ss}(E,\tau)$ in \eqref{GIT_pic_Real_quat_case}) and exploited throughout the paper to prove the Narasimhan-Seshadri correspondence (Theorem \ref{NS_over_R}). The space $\cM^{ss}(E,\tau)$ is connected (\cite{BHH}).

\subsection{The Narasimhan-Seshadri correspondence}

\subsubsection{Position of the problem} The classical approach to understanding holomorphic vector bundles on a compact Riemann surface $M$ can be summarized as follows. Any holomorphic vector bundle over $M$ admits a (unique) Harder-Narasimhan filtration (\cite{HN}) so is, in a canonical way, a successive extension of finitely many semi-stable vector bundles (with strictly increasing slopes). Subsequently, any semi-stable vector bundle $\cE$ is a successive extension of finitely many stable vector bundles of equal slope. Although the resulting filtration of $\cE$ is not unique, the associated graded object does not depend on the choice of that filtration and is, by definition, a poly-stable vector bundle. Finally, by the Narasimhan-Seshadri correspondence (\cite{NS}), a stable vector bundle of rank $r$ and degree $d$ comes, in a sense, from an irreducible $\U(r)$-representation of a certain discrete group $\Ga_d$ which is a central extension of $\piCx$ by $\Z$ (see \cite{Furuta-Steer} or Section \ref{construction_of_bundles_section} for a geometric description of $\Ga_d$ as the fundamental group of the Seifert manifold $S(L)$ over $M$, where $L$ is a smooth line bundle of degree $d$). More precisely, there is a diffeomorphism $\ModCds\simeq \cR^{\mathrm{irr}}(r,d)$, where $\cR^{\mathrm{irr}}(r,d)$ is the subset of the usual representation space $\cR(r,d):=\Hom^{\Z}(\Ga_d;\U(r))/\U(r)$ consisting of conjugacy classes of irreducible representations, as well as a homeomorphism $\ModCd\simeq\cR(r,d)$ (see Remark \ref{usual_rep_var} for the definition of $\cR(r,d)$). One way to construct that homeomorphism is given as follows. If $\rho:\Ga_d\lra \U(r)$ is a group homomorphism, then there is a diagonal action of $\Ga_d$ on $\Mt\times\C^r$, where $\Mt$ is the universal cover of $M$, $\Ga_d$ acts on $\Mt$ through the projection $\Ga_d\lra\piC$ defining  the central extension $\Ga_d$ and on $\C^r$ via the unitary representation $\rho$. The quotient $\cE_{\rho}:=(\Mt\times\C^r)/\Ga_d$ of that action is a poly-stable vector bundle of rank $r$ and degree $d$. The map $\rho\lmt\cE_{\rho}$ thus defined will be called the Narasimhan-Seshadri map. The Narasimhan-Seshadri theorem may then be formulated as follows (\cite{Don_NS}): a holomorphic vector bundle $\cE$ of rank $r$ and degree $d$ is poly-stable if and only if the corresponding complex orbit of unitary connections that it defines on an arbitrary smooth Hermitian vector bundle of rank $r$ and degree $d$ contains a projectively flat unitary connection. Such a connection is moreover unique up to unitary gauge and taking its holonomy representation induces a group homomorphism $\rho:\Ga_d\lra \U(r)$ of the appropriate type (i.e.\ lying in $\cR(r,d)$), thus providing an inverse to the Narasimhan-Seshadri map above. Now, we know from \cite{Sch_JSG} and \cite{LS} that Real (resp.\ Quaternionic) vector bundles are, in a unique way, successive extensions of semi-stable Real (resp.\ Quaternionic) vector bundles of increasing slopes and that semi-stable Real (resp.\ Quaternionic) vector bundles are extensions of finitely many stable Real (resp.\ Quaternionic) vector bundles of equal slope. So it is natural question to look for an interpretation of stable Real (resp.\ Quaternionic) vector bundles in terms of representations of a certain discrete group $\Ga_d(\Si)$. When the problem is formulated like this, finding the appropriate representation space is not an easy task, essentially because stable Real or Quaternionic vector bundles may not be geometrically stable (Proposition \ref{charac_of_stability}) and in general we cannot find filtrations of semi-stable Real and Quaternionic vector bundles whose successive quotients are geometrically stable. It is therefore better to focus either on geometrically stable Real and Quaternionic vector bundles or on poly-stable ones, as being poly-stable and Real (resp.\ Quaternionic) is equivalent to being poly-stable as a Real (resp.\ Quaternionic) bundle (Proposition \ref{ps_and_tau_comp}). We note that it is of importance for the general picture that stable Real and Quaternionic bundles be at least poly-stable when viewed as holomorphic vector bundles, which is indeed the case. At any rate, since $\ModCds\simeq\cR^{\mathrm{irr}}(r,d)$ and $\ModCds^{\Si} \simeq \ModRds \sqcup \ModHds$, one can, as a first step, look for a natural action of $\Si$ on $\cR^{\mathrm{irr}}(r,d)$ and study the fixed points of that action. The problem is that it is not immediately clear what special property these particular representations $\rho:\Ga_d\lra\U(r)$ have. When $M^\si\neq\emptyset$, we can choose a base point $x\in M^\si$, have the group $\Si$ act on $\piCx$ and its central extension $\Ga_d$, and look at $\Si$-equivariant representations $\rho:\Ga_d\lra\U(r)$ with respect to an appropriate $\Si$-action on $\U(r)$ that will depend on whether one wants to obtain Real or Quaternionic vector bundles. We will see in Theorem \ref{equivariant_version_of_NS} that this indeed gives a version of the Narasimhan-Seshadri correspondence for Real and Quaternionic vector bundles over Klein surfaces with Real points. When $d=0$, we can replace $\Ga_0$ by $\piCx$ and, in that case, it is a consequence of the results of Biswas, Huisman and Hurtubise in \cite{BHH} and Proposition \ref{equivariant_rep} of the present paper that an irreducible representation $\rho:\piCx\lra\U(r)$ whose $\U(r)$-conjugacy class is fixed under the natural $\Si$-action on $\cR^{\mathrm{irr}}(r,0)\simeq\Hom^{\mathrm{irr}}(\piCx;\U(r))/\U(r)$ defines a geometrically stable Real vector bundle if and only if it extends to a homomorphism of $\Si$-augmentations (Definition \ref{orbifold_rep}) $\chi:\pi_1(M_\Si)\lra \U(r)\rtimes\Si$, where $\pi_1(M_\Si)$ is the (orbifold) fundamental group of $(M,\si)$ (Definition \ref{def_orb_fund_gp}) and $\Si\simeq\Z/2\Z$ acts on $\U(r)$ via the involutive group automorphism $u\lmt\ov{u}$. So the group $\Ga_0(\Si)$ is essentially $\pi_1(M_\Si)$ (see Remark \ref{generalized_orbifold_rep_space} for a precise statement) and the appropriate representation space is completely identified for Real vector bundles of degree $0$. Moreover, this naturally includes the case where $M^\si=\emptyset$. For Quaternionic vector bundles of degree $0$, Biswas, Huisman and Hurtubise introduced twisted maps $\pi_1(M_\Si)\lra \U(r)\rtimes \Si$ (in particular, these are not group homomorphisms anymore) in order to deal with the condition $\tau^2=-\Id_{\cE}$ and still find a (twisted) representation space in this case. The goal of the present paper is to identify the discrete group $\Ga_d(\Si)$ for all $d\in\Z$ as well as to find the appropriate representation spaces in both the Real and Quaternionic case. We will see in Section \ref{construction_of_bundles_section} that, by contrast with the group $\Ga_d$ in the usual Narasimhan-Seshadri correspondence, the group $\Ga_d(\Si)$ is an extension of $\pi_1(M_\Si)$ by $\Z$ which is never central. For Quaternionic vector bundles, it will be appropriate to consider not the semi-direct product $\U(r)\rtimes\Si$ but a different extension $\U(r)\times_{(-1)}\Si$: the representation space that we introduce for such bundles will then consist of homomorphisms of $\Si$-augmentations $\Ga_d(\Si) \lra \U(r)\times_{(-1)}\Si$. In particular, these will indeed be group homomorphisms.

\subsubsection{Statement of the correspondence} We can now formulate the Narasimhan-Seshadri correspondence for Real and Quaternionic vector bundles to be proved in this paper. Let $(M,\si)$ be a Klein surface of topological type $(g,n,a)$ where $g\geq 2$. Take $c=+1$ for Real vector bundles and $c=-1$ for Quaternionic vector bundles and let us denote by $\cM(r,d,\vw)$ the moduli space of semi-stable Real or Quaternionic vector bundles of topological type $(r,d,\vw)$. We refer to \eqref{future_CC} for a more detailed Definition of $\cM_c(r,d,\vw)$. Let $(L,\tauL)$ be any smooth Real line bundle of degree $d$ over $(M,\si)$ and let $(S(L),\tauL)$ be the associated Real Seifert manifold, defined as in Section \ref{construction_of_bundles_section}. Denote by $\Ga_d(\Si):=\pi_1(S(L)_\Si)$ the orbifold fundamental group of $(S(L),\tauL)$. The group $\Ga_d(\Si)$ is both a non-central extension $0\lra \Z \lra \Ga_d(\Si) \lra \pi_1(M_\Si)\lra 1$ and a $\Si$-augmentation $1\lra \Ga_d \lra \Ga_d(\Si) \lra \Si\lra 1$ (see Diagram \eqref{diagram_of_ext_and_aug}). A relevant fact for the present paper is that the isomorphism class of $\Ga_d(\Si)$ as a $\Si$-augmentation depends only on $d$ (Remark \ref{ext_class_vs_aug_class}). Consider now the action $u\lmt\ov{u}$ of $\Si$ on $\U(r)$ and interpret $c=\pm 1$ as an element of $H^2(\Si;\cZ(\U(r)))\simeq\{\pm1\}$, where $\Si$ acts on the center of $\U(r)$ by complex conjugation. Then there is an extension $\U(r)\times_c\Si$ of $\Si$ by $\U(r)$, whose isomorphism class depends only on $c$ (see \eqref{gp_law_on_our_aug} for the definition of that extension). It will be convenient to call $\U(r)\times_c\Si$ the extended unitary group. Finally, consider homomorphisms of $\Si$-augmentations $\chi:\Ga_d(\Si)\lra \U(r)\times_c\Si$ such that, for all $n\in\Z\subset\Ga_d(\Si)$, $\chi(n)=\exp(i\frac{2\pi}{r}n)I_r$ and let $\U(r)$ act on such representations $\chi$ by conjugation. Let $\cR_c(r,d)$ be the associated quotient (see Definition \ref{orbifold_rep_of_extensions}) and let $\cR_c(r,d,\vw)$ be the subset of $\cR_c(r,d)$ introduced in Definition \ref{sub_repvar}: $\cR_c(r,d,\vw)$ consists of those representations $\chi$ such that the bundle associated to $\chi$ by means of Theorem \ref{constr_of_bdles_prop} is of topological type $(r,d,\vw)$. Alternately, $\cR_c(r,d,\vw)$ can be defined directly in terms of the representations $\chi$, without recurring to the notion of an associated Real or Quaternionic vector bundle, as we do in Definition \ref{sub_repvar}. A representation $\chi:\Ga_d(\Si)\lra \U(r)\times_c\Si$ whose conjugacy class lies in $\cR_c(r,d,\vw)$ will be called of type $\vw$. The main result of the paper can then be stated as follows.

\begin{theorem}[The Narasimhan-Seshadri correspondence]\label{NS_over_R}
Let $(M,\si)$ be a compact Klein surface of genus $g\geq 2$. Let $(L,\tauL)$ be a Real line bundle of degree $d$ and let $(S(L),\tauL)$ be the associated Real Seifert manifold over $(M,\si)$. Set $c=\pm 1$. Then there is a homeomorphism $$\cM_c(r,d,\vw)\simeq \cR_c(r,d,\vw)$$ between the moduli space of semi-stable Real or Quaternionic vector bundles of topological type $(r,d,\vw)$ and the space of $\U(r)$-conjugacy classes of representations of type $\vw$ of the orbifold fundamental group $\Ga_d(\Si):=\piL$ into the extended unitary group $\U(r)\times_c\Si$.
\end{theorem}

\noindent As a consequence, there are homeomorphisms $$\ModRd\simeq \cR_{+1}(r,d)\ \mathrm{and}\ \ModHd \simeq \cR_{-1}(r,d)$$ where $$\cR_c(r,d)=\Hom^{\Z}_{\Si}(\Ga_d(\Si);\U(r)\times_c\Si) / \U(r)$$ for $c=\pm1$. If we denote by $\cR^{\mathrm{irr}}_c(r,d)$ the set of homomorphisms of $\Si$-augmentations $\chi:\Ga_d(\Si)\lra \U(r)\times_c\Si$ whose restriction to $\Ga_d$ is irreducible, we obtain homeomorphisms $\mathcal{N}_c(r,d,\vw)\simeq \cR^{\mathrm{irr}}_c(r,d,\vw)$ so, in particular, $\ModRds \simeq \cR^{\,\mathrm{irr}}_{+1}(r,d)$ and $\ModHds \simeq \cR^{\,\mathrm{irr}}_{-1}(r,d)$. Moreover, Theorem \ref{NS_over_R} has the following consequence on the topology of the moduli space $\cM_c(r,d,\vw)$.

\begin{corollary}
For fixed $c$ and $(r,d,\vw)$, the homeomorphism type of the moduli space $\cM_c(r,d,\vw)$ depends only on the topological type $(g,n,a)$ of the Klein surface $(M,\si)$, i.e.\ the homeomorphism type of the surface $M/\si$, not on the complex analytic structure of $M$.
\end{corollary}

Let us end the present section by outlining the strategy of proof for Theorem \ref{NS_over_R}. In Section \ref{orbifold_groups_and_Real_bundles} we review the basics of orbifold fundamental groups and, in Theorem \ref{constr_of_bdles_prop}, we set up a map $\cR_c(r,d) \lra\cM_c(r,d)$. Then in Section \ref{moduli_of_real_and_quat_bundles} we recall the gauge-theoretic construction of the moduli spaces $\cM_c(r,d,\vw)$, completing the invariant-theoretic picture of \cite{Sch_JSG} to include results on the closure of semi-stable orbits (Theorem \ref{main_step_towards_NS}). This later enables us to construct, in Section \ref{inv_conn_and_par_transport}, representations of orbifold fundamental groups by taking the holonomy representation of a Galois-invariant, projectively flat, unitary connection on a fixed Real or Quaternionic Hermitian vector bundle (Theorem \ref{holonomy_rep}). This sets up a collection of maps $\cM_c(r,d,\vw)\lra \cR_c(r,d,\vw)$ which are all homeomorphisms (Theorem \ref{NS_over_R_fixed_top_type}), thus completing the proof of Theorem \ref{NS_over_R}. We note that, when $M^\si\neq\emptyset$, we can use Proposition \ref{equivariant_rep} to provide a $\Si$-equivariant version of the Narasimhan-Seshadri correspondence for Real and Quaternionic vector bundles, which can be formulated as follows.

\begin{theorem}[Equivariant version of \ref{NS_over_R} over Klein surfaces with Real points]\label{equivariant_version_of_NS}
Assume that $M^\si\neq\emptyset$. Choose $x\in M^\si$ and $\ov{x}\in \Fix(\tauL)$ in the fiber of $S(L)$ above $x$. Consider the action of $\Si$ on $\piCx$ given by $\ga\lmt\si\circ\ga$, as well as the induced action on the central extension $\Ga_d=\pi_1(S(L);\ov{x})$ of $\piCx$. Let $\Si$ act on $\U(r)$ either by $\si_{\R}:u\lmt\ov{u}$ or, when $r=2r'$, by $\si_{\H}:u\lmt J\ov{u}J^{-1}$, where $J$ is the usual matrix of square  $-I_{r}$. Then there are homeomorphisms $$\ModRd \simeq \Hom^{\Z}(\Ga_d;\U(r))^{\Si} / \, \O(r)$$ and $$\ModHd \simeq  \Hom^{\Z}(\Ga_d;\U(r))^{\Si} /\, \Sp(\frac{r}{2})$$ where $ \Hom^{\Z}(\Ga_d;\U(r))^{\Si}$ is the set of $\Si$-equivariant representations $\rho:\Ga_d\lra\U(r)$ satisfying, for all $n\in\Z\subset\Ga_d$, $\rho(n)=\exp(i\frac{2\pi}{r}n)I_r$, acted upon by conjugation by the group $\U(r)^\Si$, which is equal to $\O(r)$ in the Real case and to $\Sp(\frac{r}{2})$ in the Quaternionic one.
\end{theorem}

\noindent If we denote by $\Hom^{\Z}(\Ga_d;\U(r))^\Si_{\vw}$ the subset $\cW^{-1}(\vw)$ of $\Hom^{\Z}(\Ga_d;\U(r))^{\Si}$, where $\cW$ is the Real obstruction map of \eqref{real_obstruction_map}, then we also have a homeomorphism $$\cM_{\R}(r,d,\vw) \simeq \Hom^{\Z}(\Ga_d;\U(r))^\Si_{\vw}\, / \, \O(r).$$ In particular, the representation spaces $\Hom^{\Z}(\Ga_d;\U(r))^\Si_{\vw}\, / \, \O(r)$ are connected and so is $ \Hom^{\Z}(\Ga_d;\U(r))^{\Si} /\, \Sp(\frac{r}{2})$. In keeping with this equivariant perspective, we note that the Narasimhan-Seshadri map $\rho\lmt\cE_\rho$ satisfies $\si\rho\si^{-1}\lmt\ov{(\si^{-1})^*\cE_\rho}$, so the Narasimhan-Seshadri map can be made $\Si$-equivariant when $M^\si\neq\emptyset$.

\begin{ack}
It is a pleasure to thank Ahmed Abbes, Olivier Guichard, Johannes Huisman, Chiu-Chu Melissa Liu and Richard Wentworth for helpful conversations on the topics dealt with in this paper.
\end{ack}

\section{Representations of orbifold fundamental groups}\label{orbifold_groups_and_Real_bundles}

\subsection{The fundamental group of a Klein surface}\label{orbifold_groups}

Recall that a Klein surface $(M,\si)$ is in particular a topological space endowed with an action of the finite group $\Si=\Gal(\C/\R)\simeq \Z/2\Z$. Let $\Si$ act on the sphere $E\Si:=S^{\infty}$ by multiplication by $\pm 1$. This action is free with quotient $B\Si=\RP^{\infty}$ and, since $S^{\infty}$ is contractible, we have fixed a classifying bundle $E\Si \lra B\Si$ for the group $\Si$. In particular, $\pi_1(B\Si)$ is canonically identified with $\pi_0(\Si)=\Si$ via the connecting homomorphism of the homotopy long exact sequence of that fibration. The homotopy quotient of the action of $\Si$ on $M$ is, by definition, the space $M_\Si := M\times_{\Si} E\Sigma$, which is the topological quotient of $M\times E\Sigma$ by the diagonal action of $\Si$. When the action of $\Si$ on $M$ is free, $M_\Si$ is homotopically equivalent to the ordinary quotient $M/\Si$.
\begin{definition}\label{def_orb_fund_gp}
The fundamental group of the Klein surface $(M,\si)$ is the group $\pi_1(M\times_\Si E\Si;[x,1_\Si])$, which we will simply denote by $\pi_1(M_\Si;x)$ from now on.
\end{definition}
This group is known as the orbifold fundamental group of the orbifold $[M/\Si]$. It can be defined for general orbifolds, not only the good ones, but we shall not need this more general notion (\cite{Scott,Haefliger}). Given a point $x\in M$, the homotopy long exact sequence of the fibration $M\lra M_\Si \lra B\Si$ gives rise (using the fact that $M$ is connected and $\Si$ is discrete) to a short exact sequence
\begin{equation}\label{hes}
1\lra \piCx \lra \piRx \lra \Si \lra 1,
\end{equation} which we will call the homotopy exact sequence. In particular, there exists a group homomorphism $\Si \lra \Out(\piCx)$ from $\Si$ to the outer automorphism group $\Out(\piCx) := \Aut(\piCx) / \Int(\piCx)$, usually called the outer action.

To develop the theory of representations of fundamental groups of Klein surfaces, it will be useful to formulate it in terms of group augmentations.
\begin{definition}
An augmentation of the group $\Si$ (or $\Si$-augmentation) is a surjective group homomorphism $\alpha:G(\Si)\lra \Si$. A homomorphism between two $\Si$-augmentations $\alpha_1:G_1(\Si)\lra \Si$ and $\alpha_2:G_2(\Si)\lra \Si$ is a group homomorphism $\phi:G_1(\Si) \lra G_2(\Si)$ such that the following diagram commutes:
$$\begin{CD}
G_1(\Si) @>{\alpha_1}>> \Si \\
@VV{\phi}V  \| \\
G_2(\Si) @>{\alpha_2}>> \Si.
\end{CD}$$
\end{definition}
\noindent Note that, if $\phi$ is a homomorphism of $\Si$-augmentations, then $\phi(\ker\alpha_1) \subset \ker \alpha_2$. We will sometimes use the notation $G:=\ker\alpha$ when $\alpha:G(\Si) \lra\Si$ is a $\Si$-augmentation. The notion of representation of the fundamental group of a Klein surface then goes as follows.
\begin{definition}\label{orbifold_rep}
Let $\alpha: G(\Si)\lra \Si$ be a $\Si$-augmentation and denote by $G$ the kernel of $\alpha$. A representation of $\piRx$ in $G(\Si)$ is a homomorphism of $\Si$-augmentations, i.e.\ a group homomorphism $\chi:\piRx \lra G(\Si)$ making the following diagram commutative:
$$\begin{CD}
1 @>>> \piCx @>>> \piRx @>>> \Si @>>> 1 \\
@. @VVV @VV{\chi}V \| @.\\
1 @>>> G @>>> G(\Si) @>>> \Si @>>> 1.
\end{CD}$$ The set of such representations will be denoted by $\Hom_\Si(\piRx;G(\Si))$. When $G$ is a Lie group, $G(\Si)$ has an induced Lie group topology, which in turn causes $\Hom_\Si(\piRx;G(\Si))$ to inherit a topology.

Two representations $\chi_1$ and $\chi_2$ are called equivalent if there exists an element $g\in G$ such that $\chi_2 = \Int_g\circ \chi_1$, where by $\Int_g$ we denote conjugation by $g\in G$ inside $G(\Si)$. The associated representation space is the space \begin{equation}\label{orbifold_rep_var}\Hom_\Si(\piRx;G(\Si))/G,\end{equation} endowed, when $G$ is a Lie group, with the quotient topology. In particular, when $G$ is compact, this representation space is Hausdorff.
\end{definition} 

It is worth noting that the representation space in \eqref{orbifold_rep_var} is independent of the choice of the base point $x\in M$ since, if we choose a different base point $x'\in M$, we obtain an isomorphic $\Si$-augmentation by picking a path from $x$ to $x'$ and the difference of two such paths is an element of $\piCx$, so is eventually absorbed into the $G$-action. Also, if we choose a different model for $E\Si$, the short exact sequence \eqref{hes} will be replaced by a non-canonically isomorphic one but the representation space \eqref{orbifold_rep_var} will remain the same. In \eqref{alternate_def_pi_1_orb}, an alternate definition of $\piRx$ is given which is independent of the choice of a particular model for $E\Si$.

We will often use the notation $\rho:=\chi|_{\piCx}$ for the restriction of $\chi$ to $\piCx$ and $\chi=\widehat{\rho}$ for the extension of a group homomorphism $\rho:\piCx\lra G$ to a homomorphism of $\Si$-augmentations $\chi:\piRx\lra G(\Si)$. The fact that, in the notion of equivalence between representations of $\piRx$, the element $g$ is required to lie  in $G$ and not in $G(\Si)$ implies the existence of a map\begin{equation}\label{forgetful_map}\begin{array}{ccc}\Hom_\Si(\piRx;G(\Si)) / G & \lra & \Big( \Hom(\piCx;G)/G \Big)^\Si \\ {[}\chi{]} & \lmt & {[}\chi|_{\piCx}{]} \end{array}\end{equation} which is, however, neither injective nor surjective in general (for example in the present paper). Note that the target space of \eqref{forgetful_map} is the fixed-point set of the natural $\Si$-action on the $G$-representation space of $\piCx$, obtained by composing the outer action $$\Si \lra \Out(G) \times \Out(\piCx)$$ with the usual $\Out(G) \times \Out(\piCx)$-action on the $G$-representation space of $\piCx$ (the latter being defined by $[g,\beta]\cdot[\rho] = [\Int_g\circ\rho\circ\beta^{-1}]$). 

An important example of representation of $\piRx$ is given as follows. Let $\Mt(x)$ denote the set of homotopy classes of continuous paths in $M$ whose starting point is $x$. The canonical projection $q:\Mt(x)\lra M$ (taking a path $\eta$ to its ending point) turns $\Mt(x)$ into the universal covering space of $M$. Since the action of $\Si$ on $M\times E\Si$ is free, the quotient map $p:(M\times E\Si)\lra M_\Si$ is a covering map and the simply connected space $\Mt(x)\times E\Si$ is the universal covering space of $M_\Si$. Therefore, by the classical theory of covering spaces, we have an isomorphism of $\Si$-augmentations
$$\begin{CD}
1 @>>> \piCx @>>> \piRx @>>> \Si @>>> 1 \\
@. @VVV @VVV \| @.\\
1 @>>> \Aut\big(\Mt(x)/M\big) @>>> \Aut\big((\Mt(x)\times E\Si)/M_\Si\big) @>{\alpha}>> \Si @>>> 1
\end{CD}$$ where the augmentation map $\alpha$ sends the deck transformation $f\in\Aut((\Mt(x)\times E\Si)/M_\Si)$ to the induced deck transformation $\si_f\in \Aut((M\times E\Si)/M_\Si) \simeq \Si$ of the intermediate covering space $p:(M\times E\Si) \lra M_\Si$ between $\Mt(x)\times E\Si$ and $M_\Si$. Moreover, it is straightforward to check that the $\Si$-augmentation in the bottom line of the diagram above is isomorphic to $$1\lra \Aut(\Mt(x)/M) \lra \Aut_\Si (\Mt(x)/M) \lra \Si \lra 1$$ where \begin{equation}\label{alternate_def_pi_1_orb}\Aut_{\Si}(\Mt(x)/M) := \{h: \Mt(x)\lra \Mt(x)\ |\ \exists\ \si_{h}\in\Si, q\circ h=\si_h\circ q\},\end{equation} and the augmentation homomorphism is the map $h\lmt\si_h$. In particular, there is an isomorphism $\piRx\simeq\Aut_\Si(\Mt(x)/M)$ and the orbifold fundamental group of $[M/\Si]$ acts on $\Mt(x)$, the universal covering space of $M$, by generalized deck transformations (covering either $\Id_M$ or $\si:M\lra M$), as pictured in the following diagram. 
$$
\begin{CD}
\Mt(x) \times E\Si @>{f=h\otimes\si_h}>> \Mt(x) \times E\Si \\
@VV{q\otimes \Id}V @VV{q\otimes \Id}V \\
M \times E\Si @>{\si_f=\si_h\otimes\si_h}>> M \times E\Si \\
@VV{p}V @VV{p}V \\
M_{\Si} @= M_{\Si}
\end{CD}
$$
Yet another possible characterization of the orbifold fundamental group is the following description in terms of paths in $M\times E\Si$, which will be useful in Section \ref{parallel_transport}. Since the map $p:M\times E\Si \lra M_\Si$ is a covering map and we have chosen a point $x\in M$, a loop $\eta$ based at $[x,1_\Si]$ lifts unically to a path $\etat^{(x,1_\Si)}$ in $M\times E\Si$ satisfying $\etat^{(x,1_\Si)}(0)=(x,1_\Si)$. One then has $\etat^{(x,1_\Si)}(1)=\lambda_\eta^{-1}(x,1_\Si)$, where $\lambda_\eta\in\Si$ is the image of $\eta\in\piRx$ under the augmentation homomorphism. By projecting the path $\etat^{(x,1_\Si)}$ to $M$, we obtain a path $\ga_\eta$ in $M$, going from $x$ to $\lambda_\eta^{-1}(x)$. Let us now consider the set \begin{equation}\label{paths_description}\widetilde{P}_x:=\{(\ga,\lambda):\ga:[0;1]\lra M, \lambda\in\Si, \ga(0)=x, \ga(1)=\lambda^{-1}(x)\}\end{equation} endowed with the composition law \begin{equation}\label{composition_of_paths}(\ga_1,\lambda_1)(\ga_2,\lambda_2)=(\lambda_2^{-1}(\ga_1)\ga_2,\lambda_1\lambda_2),\end{equation} where the product $\lambda_2^{-1}(\ga_1)\ga_2$ is the ordinary composition of paths in $M$ (from right to left).  Then the set $P_x:=\widetilde{P}_x/\mathrm{(homotopy\ in}\ M\mathrm{)}$ is a group and the map $P_x\lra\Si$ sending $(\ga,\lambda)$ to $\lambda$ is an augmentation homomorphism, with kernel $L_x:=\widetilde{L_x}/\mathrm{(homotopy\ in}\ M\mathrm{)}$, where $$\widetilde{L}_x=\{\ga:[0;1]\lra M,\ga(0)=\ga(1)=x\}$$ Finally, the map $\eta \lmt \etat^{(x,1_\Si)} \lmt (\ga_\eta,\lambda_\eta)$ constructed above sets up an isomorphism of $\Si$-augmentations
$$\begin{CD}
1 @>>> \piCx @>>> \piRx @>>> \Si @>>> 1 \\
@. @VVV @VVV \| @.\\
1 @>>> L_x @>>> P_x @>{\alpha}>> \Si @>>> 1
\end{CD}$$ 

\noindent In the remainder of the paper, we will often omit the base point $x\in M$ and simply write $\pi_1(M)$, $\pi_1(M_\Si)$ and $\Mt$ in place of $\piCx$, $\piRx$ and $\Mt(x)$.

\subsection{Construction of Real and Quaternionic vector bundles}\label{construction_of_bundles_section}

We assume from now on that $M$ is compact, of genus $g\geq 1$. Let $(L,\tauL)$ be a $C^\infty$ Real line bundle over $(M,\si)$ and consider the smooth complex surface $L\setminus\{0_{L}\}$ obtained from $L$ by removing the zero section. This is preserved by $\tau_L$ and deformation retracts onto a $\Si$-equivariant circle bundle over $M$ (to construct it, we can for instance pick up a $\Si$-invariant Hermitian metric on $L$ and consider the unit bundle of $L$: this is indeed preserved by $\tauL$). Let us denote by $S(L)$ the $S^1$-bundle thus obtained (by definition, this is a Seifert manifold over $M$). The action of $\Si$ on $S(L)$ makes it possible to consider the orbifold fundamental group of the Real Seifert manifold $(S(L),\tauL)$. Allowing ourselves to omit base points from the notation, we have the following commutative diagram
\begin{equation}\label{diagram_of_ext_and_aug}
\begin{CD}
@. 0 @. 0 \\
@. @VVV @VVV \\
@. \Z @= \Z \\
@. @VVV @VVV \\
1 @>>> \pi_1(S(L)) @>>> \pi_1(S(L)_\Si) @>>> \Si @>>> 1 \\
@. @VVV @VVV \| \\
1 @>>> \pi_1(M) @>>> \pi_1(M_\Si) @>>> \Si @>>> 1\\
@. @VVV @VVV \\
@. 1 @. 1
\end{CD}
\end{equation} where $\Z\simeq\pi_1(S^1)$ is the fundamental group of the fiber of $S(L)$ and $S(L)_{\Si}$ is, as in Section \ref{orbifold_groups}, the homotopy quotient $S(L)\times_{\Si}E\Si$. In the first column, the group $\pi_1(S(L))$ is a central extension of $\pi_1(M)$ by $\Z$ while, in the second column, the extension $\pi_1(S(L)_\Si)$ of $\pi_1(M_\Si)$ by $\Z$ is non-central. Indeed, by a result of Kahn (\cite{Kahn}), smooth Real line bundles over a Klein surface $(M,\si)$ are classified by their equivariant first Chern class $c_1^{\Si}(L,\tauL)\in H^2_{\Si}(M;\underline{\Z})$ where $\Si$ acts on $M$ by the orientation-reversing involution $\si$ and on the local $\underline{\Z}$-coefficients by $n\lmt -n$. Since we are assuming that $g\geq 1$, the space $M_\Si$ is an Eilenberg-MacLane space $K(\pi;1)$ so $H^2_{\Si}(M;\underline{\Z})\simeq H^2(\pi_1(M_\Si);\Z)$, where $\pi_1(M_\Si)$ acts on the Abelian group $\Z$ via the augmentation homomorphism $\pi_1(M_\Si)\lra \Si$ and the previous (non-trivial) action of $\Si$ on $\Z$, meaning that there is a bijection between the set of isomorphism classes of smooth Real line bundles over $(M,\si)$ and extensions of $\pi_1(M_\Si)$ by $\Z$ inducing the action $n\lmt -n$ on $\Z$: that bijection is precisely given by the map taking a given Real line bundle $(L,\tauL)$ to the orbifold fundamental group $\pi_1(S(L)_\Si)$ of the associated Real Seifert manifold. Since the action of $\pi_1(M_\Si)$ on $\Z$ in that construction is non-trivial, the extension $\pi_1(S(L)_\Si)$ of $\pi_1(M_\Si)$ by $\Z$ is indeed non-central. Note that the sub-group $\pi_1(M)$ of $\pi_1(M_\Si)$, however, acts trivially on $\Z$, because it maps trivially to $\Si$ through the augmentation homomorphism, so the extension $\pi_1(S(L))$ of $\pi_1(M)$ by $\Z$ is central: it is the central extension corresponding to the ordinary first Chern class $c_1(L)\in H^2(M;\Z) \simeq H^2(\pi_1(M);\Z)$, as observed by Furuta and Steer in \cite{Furuta-Steer}.

We now wish to construct Real and Quaternionic vector bundles over the Klein surface $(M,\si)$ starting from certain representations of the orbifold fundamental group of the Real Seifert manifold $(S(L),\tau_L)$. Let us fix a smooth Real line bundle $(L,\tauL)$ over $(M,\si)$ and denote its degree by $d:=\int_{M} c_1(L) \in\Z$. In order to define the appropriate representation space for $\piLR$, we need to specify a $\Si$-augmentation $1\lra G \lra G(\Si) \lra \Si \lra 1$, where $G$ very soon will be the unitary group $\U(r)$. Such a $\Si$-augmentation in particular defines an outer action $\Si\lra\Out(G)$ and we will always assume that $\Si$ actually acts on $G$, in other words that there is an given lifting of that outer action to $\Aut(G)$. The typical example for us will be the $\Si$-action $\si:u\lmt\ov{u}$ on $\U(r)$. When an outer action $\Si\lra \Out(G)$ comes from a fixed action of $\Si$ on $G$, isomorphism classes of extensions of $\Si$ by $G$ which induce that outer action are in bijective correspondence with the group $H^2(\Si;\cZ(G))$, where $\cZ(G)$ is the center of $G$. If we represent a class $[c]\in H^2(\Si;\cZ(G))$ by a normalized $2$-cocycle $c:\Si\times\Si \lra \cZ(G)$, then the corresponding extension $G(\Si)$ may be viewed as the product set $G\times\Si$ equipped with the group law defined by the relation
\begin{equation}\label{gp_law_on_our_aug}
(c(\si_1,\si_2),1) (g_1,\si_1)(g_2,\si_2) = (g_1\si_1(g_2),\si_1\si_2).
\end{equation} We will denote that extension by $G\times_c\Si$. Its isomorphism class depends only on $[c]$ and, if $c$ is a coboundary, then $G\times_c\Si$ is isomorphic to the semi-direct product $G \rtimes \Si$ defined by the given action of $\Si$ on $G$. Since in this paper we are assuming that $\Si\simeq\Z/2\Z$, one may observe that
\begin{equation}\label{obs_on_H2}
H^2(\Si;\cZ(G)) \simeq \cZ(G)^\Si / \{a\si(a) : a\in \cZ(G)\}
\end{equation} In particular, the normalized $2$-cocycle $c$ is then entirely determined by its value $c(\si,\si)\in \cZ(G)^\Si$, where $\si$ is the non-trivial element of $\Si$. So in practice we can think of $c$ as an element of $\cZ(G)^\Si$. If for instance $G=\U(r)$ and $\si(u)=\ov{u}$, then $\cZ(G)=S^1$ so, if $a\in \cZ(G)$, $\si(a)=a^{-1}$. Therefore, by Observation \eqref{obs_on_H2}, one has $H^2(\Si;\cZ(G)) = \{\pm1\}$. From now on, we fix $G=\U(r)$ and we consider the $\Si$-action $\si:u\lmt\ov{u}$ on $\U(r)$. Moreover, we assume that an extension class $c=\pm1\in H^2(\Si;S^1)$ has been fixed.

Next, we need to specify what kind of representations of $\piLR$ we will be considering. Those will be homomorphisms of $\Si$-augmentations
\begin{equation}\label{morphisms_of_extensions}
\begin{CD}
1 @>>> \piL @>>> \piLR @>>> \Si @>>> 1\\
@. @VV{\chi|_{\piL}}V @VV{\chi}V \| \\
1 @>>> \U(r) @>>> \U(r)\times_c \Si @>>> \Si @>>> 1
\end{CD}
\end{equation} satisfying the additional requirement \begin{equation}\label{cond_on_morphisms}\forall\,n\in\Z, \chi(n) = e^{i\frac{2\pi}{r}\,n}I_r \in \cZ(\U(r)).\end{equation} Condition \eqref{cond_on_morphisms} means that the homomorphism of $\Si$-augmentations $\chi$ also makes the following diagram commute:
$$\begin{CD}
0 @>>> \Z @>>> \piLR @>>> \piR @>>> 1\\
@. @VV{n\mapsto\exp(i\frac{2\pi}{r}n)}V @VV{\chi}V @VVV \\
1 @>>> S^1 @>>> \U(r)\times_c \Si @>>> \PU(r)\rtimes\Si @>>> 1
\end{CD}$$ where $\Si$ acts on $\PU(r)$ by complex conjugation. 

\begin{definition}\label{orbifold_rep_of_extensions} The set of all homomorphisms of $\Si$-augmentations which satisfy Condition \eqref{cond_on_morphisms} will be denoted by $\Hom_{\Si}^{\Z}(\piLR;\U(r)\times_c\Si)$ and endowed with the $\U(r)$-action $u\cdot\chi = \Int_u\circ\chi$, thus providing a notion of equivalence of representations, as in Definition \ref{orbifold_rep}. The associated representation space is the quotient space \begin{equation}\label{our_rep_space}\mathcal{R}_c(r,d) := \Hom_{\Si}^{\Z}(\piLR;\U(r)\times_c\Si) / \U(r).\end{equation}
\end{definition}

\begin{remark}\label{usual_rep_var}
Let $\cR(r,d):=\Hom^\Z(\piL;\U(r))/\U(r)$ be the usual representation space for $\piL$, to be useful in Proposition \ref{equivariant_rep} and consisting of group homomorphisms $\rho:\piL\lra\U(r)$ satisfying, for all $n\in\Z$, $\rho(n)=\exp(i\frac{2\pi}{r}n)I_r$. Similarly to \eqref{forgetful_map}, there is a natural $\Si$-action on $\cR(r,d)$ and the map $\chi\lmt\chi|_{\piL}$ is a map from $\cR_c(r,d)\lra\cR(r,d)^\Si$, which is not surjective in general.
\end{remark}

To finish the construction of Real and Quaternionic vector bundles, we pick an element $\sit\in\piLR$ in the fiber above the non-trivial element $\si\in \Si$. Note that any two choices of such a lifting $\sit$ of $\si\in\Si$ differ by an element  of $\piL$. The element $\sit$ acts on $\Mt$ via the group homomorphism $\piLR\lra\piR\simeq\Aut_{\Si}(\Mt/M)$ (although $\sit^2\neq\Id_{\Mt}$ in general). The group $\Si$ acts on $\C^r$ via $\si:z\lmt\ov{z}$ (complex conjugation on each coordinate of the vector $z$) and on $\U(r)$ via $\si:u\lmt\ov{u}$. In particular, for all $u\in\U(r)$ and all $z\in\C^r$, $\si(uz)=\si(u)\si(z)$. Consider then any representation $\chi:\piLR\lra\U(r)\times_c\Si$ in the sense of \eqref{our_rep_space} and denote its restriction to $\piL$ by $\rho:=\chi|_{\piL}:\piL\lra\U(r)$. Let us  define $\usit\in\U(r)$ by the equation $\chi(\sit)=(\usit,\si)$ and equip the product bundle $\Mt\times\C^r$ over $\Mt$ with the diagonal $\piL$-action $\ga\cdot(\delta,v) := (\ga\cdot\delta, \rho(\ga) v)$, defined by the homomorphism $\piL\lra\piC\simeq\Aut(\Mt/M)$ and by $\rho$. This in particular defines a rank $r$ holomorphic vector bundle $$\cE_\rho=\Mt\times_{\rho}\C^r := (\Mt\times\C^r)/\piL$$ over $M$, whose determinant is smoothly isomorphic to $L$ so in particular $\deg\cE_{\rho}=d$ (if Condition \eqref{cond_on_morphisms} were modified to $n\lmt \exp(i\frac{k2\pi}{r}n)I_r$ for some fixed $k\in\Z$, then $\det(\cE_\rho)$ would become smoothly isomorphic to $L^k$ and have degree $kd$, \cite[Theorem 4.1]{Furuta-Steer}). Finally, consider the map 
\begin{equation}\label{def_of_real_structure}
\tausit:
\begin{array}{ccc}
\Mt\times\C^r & \lra & \Mt\times\C^r \\
(\delta,v) & \lmt & \big(\sit\cdot\delta,\usit\si(v)\big)
\end{array}.
\end{equation}
\begin{theorem}\label{constr_of_bdles_prop}
The map $\tausit$ induces an anti-holomorphic map $\tausi$ on the holomorphic vector bundle $\cE_\rho=\Mt\times_\rho\C^r$ over $M$. The map $\tau$ on $\cE_\rho$ lies above the map $\si$ on $M$, is fiberwise $\C$-anti-linear and squares to $c\,\Id_{\cE}$. In other words, $(\cE_{\rho},\tausi)$ is a Real vector bundle if $c=1$ and a Quaternionic vector bundle if $c=-1$. Moreover:
\begin{itemize}
\item Equivalent representations $\chi$ and $\chi'$ give rise to isomorphic Real or Quaternionic vector bundles $(\cE_\rho,\tausi)$ and $(\cE_{\rho'},\tausi')$.
\item A different choice of $\sit'\in\piLR$ in the fiber of the augmentation map above $\si\in\Si$ gives rise to a Real or Quaternionic vector bundle $(\cE_\rho,\tausi')$ isomorphic to $(\cE_\rho,\tausi)$.
\end{itemize}
\end{theorem}
\begin{proof} The only things to check are that $\tausit$ descends to $\cE_\rho$ and that the induced map $\tausi$ squares to $c\,\Id_{\cE}$.

If $(\delta,v)\in (\Mt\times\C^r)$ and $\ga\in\piC$, we have
\begin{eqnarray*}
\tausit\big(\ga\cdot(\delta,v)\big) & = & \tausit\big(\ga\cdot\delta,\rho(\ga) v\big)\\
& = & \big(\sit\big(\ga\cdot\delta\big),u_{\sit}\si\big(\rho(\ga) v)\big)\\
& = & \big(\sit\ga\sit^{-1}\cdot\big(\sit\cdot\delta\big),\usit\si\big(\rho(\ga)\big)\usit^{-1}\big(\usit\si(v)\big)\big)\\
& \sim & \big(\sit\cdot\delta,\usit\si(v)\big) = \tausit(\delta,v)
\end{eqnarray*} since $\sit\ga\sit^{-1}\in\piC$ and, using Equation \eqref{gp_law_on_our_aug},
\begin{eqnarray*}
\big(\rho(\sit\ga\sit^{-1}),1\big) =\chi\big(\sit\ga\sit^{-1}\big) & = & \chi(\sit)\chi(\ga)\chi(\sit)^{-1}\\
& = & \big(\usit,\si\big)\big(\rho(\ga),1\big)\big(\usit,\si\big)^{-1}\\
& = & \big(\usit\si\big(\rho(\ga)\big),\si\big)\big(c\si(\usit^{-1}),\si\big)\\
& = & \big(\usit\si\big(\rho(\ga)\big)u_{\sit}^{-1},1\big)
\end{eqnarray*} so $\tausit$ indeed induces a map $\tausi$ on $\cE_{\rho}=\Mt\times_{\rho}\C^r$.

Moreover, one has $\tausit^2\big(\delta,v\big) =  \big(\sit^2\cdot\delta,\usit\si(\usit)v\big) \sim (\delta,c\cdot v)$ since, using again Equation \eqref{gp_law_on_our_aug}, $\chi(\sit^2) = \chi(\sit)\chi(\sit) = (\usit,\si)(\usit,\si) = (c^{-1}\usit\si(\usit),\si^2)$ so the induced map $\tausi$ indeed satisfies $\tausi^2=c\,\Id_{\cE}$.
\end{proof}

\noindent By the Narasimhan-Seshadri theorem (\cite{NS}), the vector bundle $\cE_\rho$ is poly-stable and of rank $r$ and degree $d$. What Theorem \ref{constr_of_bdles_prop} says is that, if $\rho:\piL\lra\U(r)$ extends to a representation $\widehat{\rho}:\piLR\lra \U(r)\times_c\Si$, then $\cE_\rho$ admits a natural Real structure if $c=1$ and a natural Quaternionic structure if $c=-1$. One of the goals of the present paper is to show that, conversely, any poly-stable Real or Quaternionic vector bundle of rank $r$ and degree $d$ on $(M,\si)$ is isomorphic to a bundle $(\cE_\rho,\tau)$ obtained as above from a rank $r$ representation $\widehat{\rho}$ of $\piLR$, where $(L,\tauL)$ is any smooth Real line bundle of degree $d$ on $(M,\si)$.

\begin{remark}\label{generalized_orbifold_rep_space}
When $L=L_0:=M\times\C$, the trivial Real line bundle on $M$, one has $\pi_1(S(L_0)_\Si) \simeq \Z\rtimes\piR$, where $\piR$ acts on $\Z$ through the augmentation homomorphism to $\Si$ and the non-trivial action of $\Si$ on $\Z$. In particular, $\piR$ is now a sub-group of $\pi_1(S(L_0)_\Si)$, so there is a well-defined map $$\begin{array}{ccc}
\cR_c(r,d) & \lra & \Hom_{\Si}(\piR;\U(r)\times_c\Si) / \U(r) \\
\chi & \lmt & \chi|_{\piR}
\end{array}$$ which is in fact a homeomorphism whose inverse map is given as follows: to any homomorphism of $\Si$-augmentations $\chi_0:\piR\lra \U(r)\times_c\Si$, we associate the map $$\chi: \begin{array}{ccc}
\Z\rtimes\piR & \lra & \U(r)\times_c\Si\\
(n,\eta) & \lmt & (e^{i\frac{2\pi}{r}n},1_\Si)\,\chi_0(\eta)
\end{array}$$ which is readily checked to be a homomorphism of $\Si$-augmentations such that $\chi(n,1)=\exp(i\frac{2\pi}{r}n)I_r$ for all $n$ in $\Z$ and satisfying $\chi|_{\piR}=\chi_0$. Thus, the representation spaces $\mathcal{R}_c(r,d)$ constructed in \eqref{our_rep_space} are a generalization of the representation space of $\piR$ introduced in \eqref{orbifold_rep_var}. 
\end{remark}

To sum up, on a Klein surface $(M,\si)$, we can construct, given any smooth Real line bundle, an orbifold fundamental group $\pi_1(S(L)_\Si)$, which is a non-central extension of $\pi_1(M_\Si)$ by $\Z$, as well as representation spaces $\cR_c(r,d)$ for $\pi_1(S(L)_\Si)$. Theorem \ref{constr_of_bdles_prop} then says that there are maps 
\begin{equation}\label{map_from_reps_to_bundles} 
\cR_{+1}(r,d) \lra \ModRd\quad \mathrm{and}\quad \cR_{-1}(r,d)\lra \ModHd
\end{equation} where $\ModRd$ (resp\ $\ModHd$) denotes the moduli space of poly-stable Real (resp.\ Quaternionic) vector bundles of rank r and degree d (see Definition \ref{def_moduli_spaces_Real_Quat_bdles}). Theorem \ref{NS_over_R} implies that these maps are homeomorphisms.

\begin{remark}\label{ext_class_vs_aug_class}
Note that isomorphic smooth Real line bundles give rise to isomorphic extensions $\piLR$ of $\piR$ by $\Z$ and non-isomorphic smooth Real line bundles give rise to non-isomorphic extensions but that, as an augmentation of $\Si$, $\piLR$ only depends on the degree $d$ of $L$, not on the whole $c_1^{\Si}(L,\tau_L)$, because the extension class of the central row in Diagram \eqref{diagram_of_ext_and_aug} is determined by a class in $H^2(\Si;\cZ(\piL)) \simeq H^2(\Si;\Z)$, which is trivial when $\Si$ acts on $\Z$ by $n\lmt -n$ (see \eqref{obs_on_H2}). As a consequence, the topology of $\cR_c(r,d)$ only depends on the topological invariants $(g,n,a)$ of the Klein surface $(M,\si)$, the topological invariants $(r,d)$ of complex vector bundles over $M$ and the extension class $c=\pm 1$ that we consider. We shall see in Section \ref{CC} how to use additional topological invariants to distinguish the connected components of $R_c(r,d)$.
\end{remark}

\subsection{The case of a Klein surface with Real points}\label{case_with_real_points}

When the Klein surface $(M,\si)$ satisfies $M^\si\neq\emptyset$, the construction of Real and Quaternionic vector bundles given in Theorem \ref{constr_of_bdles_prop} can be simplified by choosing the base point $x\in M$ to lie in $M^\si$. By functoriality of the orbifold fundamental group, the choice of $x\in M^\si$ makes $\piRx$ isomorphic, as an extension of $\Si$ by $\piCx$, to the semi-direct product $\piCx\rtimes\Si$ for the $\Si$-action on $\piCx$ defined by $\ga\lmt\si\circ\ga$, which we simply denote by $\si(\ga)$. Since the Real Seifert manifolds constructed in Section \ref{construction_of_bundles_section} always have Real points when $M$ does, the orbifold fundamental group $\piLR$ is also isomorphic to a semi-direct product $\piL\rtimes\Si$, via the choice of a point $\ov{x}\in\Fix(\tauL)$ in the fiber of $S(L)$ above $x$. So the representation spaces \eqref{our_rep_space} now involve those semi-direct products $\piL\rtimes\Si$, as well as the extensions $\U(r)\times_c\Si$ for $c=\pm 1$. When $c=1$, $\U(r)\times_c\Si$ is isomorphic to the semi-direct product $\U(r)\rtimes_{\si_\R}\Si$, where $\si_\R:u\lmt\ov{u}$ on $\U(r)$. When $c=-1$, $\U(r)\times_{c}\Si$ is not a semi-direct product for the $\Si$-action $u\lmt\ov{u}$ on $\U(r)$ but it can be sometimes be made into a semi-direct product for a different lifting of the same outer action $\Si\lra\Out(\U(r))$, namely for the $\Si$-action $\si_\H:u\lmt J\ov{u}J^{-1}$ when $r=2r'$ and $J=\begin{bmatrix} 0 & -I_{r'} \\ I_{r'} & 0\end{bmatrix}$, which indeed differs from $\si_\R:u\lmt \ov{u}$ only by the inner automorphism $\Int_{J}$. Note that the purpose of doing that is to handle Quaternionic vector bundles over $(M,\si)$ and that the latter always have even rank $r=2r'$ when $M^\si\neq\emptyset$, so the parity condition on $r$ is not a restriction. Moreover, when $c=-I_r\in \cZ(\U(r))^\Si$, one has $c=J^2=J\si(J)$, so the cocycle $c=-1$ can be written under the form $a\si(a)$ for some $a\in G$ (this means that $c$ is a $\cZ(G)$-valued $2$-cocycle which, in a sense, splits over $G$ but not over $\cZ(G)$, compare \eqref{obs_on_H2}). This readily implies that the map $$\begin{array}{ccc}
\U(2r')\times_{(-1)}\Si & \lra & \U(2r')\rtimes_{\si_\H}\Si\\
(u,\eps) & \lmt & \big((JuJ^{-1})a_\eps,\eps\big)
\end{array}$$ where $a_{1_\Si} = I_{2s}$ and $a_\si=J$ ($a$ is a $G$-valued normalized $1$-cochain) is an isomorphism of $\Si$-augmentations. 
Observe that the two extensions of $\Si$ by $\U(2r')$ above indeed induce the same outer action $\Si\lra\Out(\U(2r'))$, namely the outer action defined by the inner equivalence class of $\si_\R$, as well as the same $\Si$-action on $\cZ(\U(2r'))\simeq S^1$, namely $z\lmt\ov{z}$.

Summing up, when $M^\si\neq\emptyset$, the representation spaces $\cR_c(r,d)$ of \eqref{our_rep_space} become $$\Hom_{\Si}^{\Z}(\piL\rtimes\Si;\U(r)\rtimes_{\si}\Si) / \U(r)$$ for the $\Si$-action on $\U(r)$ defined by either $\si=\si_\R$ (if $c=+1$) or $\si=\si_\H$ (if $c=-1$). It turns out that the involutions $\si_\R$ and $\si_\H$ share the following property: \begin{equation}\label{connectedness_ppty} \si(u)=u^{-1}\ \mathrm{if\ and\ only\ if}\ \exists\ v\in\U(r), u=v\si(v^{-1})\end{equation} (when $\si=\si_\R$, this simply means that a symmetric unitary matrix $u$ is of the form $vv^t$, which is a classical result; the case of $\si_\H$ is similar) and that Property \eqref{connectedness_ppty} is a sufficient condition to remove the semi-direct products in the above representation spaces, in the following sense.

\begin{proposition}\label{equivariant_rep}
Assume that $M^\si\neq\emptyset$ and that $\Si$ acts on $\U(r)$ by $\si_\R:u\lmt\ov{u}$ or, when $r=2r'$, either by $\si_\R$ or by $\si_\H:u\lmt J\ov{u}J^{-1}$. Then there is a homeomorphism 
\begin{equation*} 
\cR_c(r,d) \simeq \Hom^{\Z}(\piL;\U(r))^\Si /  \U(r)^\Si \, ,
\end{equation*} where on the 
right-hand side
we consider all $\Si$-equivariant representations $$\rho:\piL\lra\U(r)$$ such that, for all $n\in\Z\subset\piL$, $\rho(n)=\exp(i\frac{2\pi}{r}n)I_r$, up to $\U(r)^\Si$-conjugation.
\end{proposition}

\noindent Note that $\U(r)^\Si=\O(r)$ if $\si=\si_\R$ and $\U(r)^\Si = \Sp(\frac{r}{2})$ if $\si=\si_\H$ and that the representation space on the 
right-hand
side does not depend on the choice of the base point $x\in M^\si$ since, by the observation after Definition \ref{orbifold_rep}, the representation space $\cR_c(r,d)$ 
is independent of such a choice.

The proof of Proposition \ref{equivariant_rep} is based on an argument in group cohomology. Recall that we are assuming that there is a fixed action of $\Si$, which could in fact be any finite group, on the given Lie group $G=\U(r)$, which could also be arbitrary, by group automorphisms. The pointed set $H^1(\Si;G)$ is by definition the set of all crossed homomorphisms, i.e.\ (normalized) maps $a:\Si\lra G$ such that, for all $(\si,\si')\in\Si\times\Si$, $a_{\si\si'}=a_{\si}\si(a_{\si'})$, modulo the following equivalence relation: $a\sim a'$ if there exists $g\in G$ such that, for all $\si\in\Si$, $a'_\si=ga_\si\si(g^{-1})$. If the action of $\Si$ on $G$ is trivial, then $H^1(\Si;G)=\Hom(\Si;G)/G$. Also, when $\Si\simeq\Z/2\Z$, a normalized $1$-cocycle $a$ is determined by its value on the non-trivial element $\si\in\Si$ (so we can identify $a_\si$ and $a$) and we have $H^1(\Z/2\Z;G) \simeq \{a\in G\ |\ \si(a)=a^{-1}\} / G$, where the $G$-action on the set of elements $a\in G$ satisfying $\si(a)=a^{-1}$ is given by $g\cdot a=g a\si(g^{-1})$, which again is ordinary conjugacy if the $\Si$-action on $G$ is trivial so, in that case, $H^1(\Z/2\Z;G)$ is the set of conjugacy classes of order $2$ elements of $G$. We then see that Property \eqref{connectedness_ppty} amounts to saying that $H^1(\Si;\U(r))=\{1\}$ when $\Si\simeq\Z/2\Z$ acts on $\U(r)$ by $u\lmt\ov{u}$ or by $u\lmt J\ov{u}J^{-1}$ and we can now give a proof of Proposition \ref{equivariant_rep} which is valid for an arbitrary Lie group $G$ acted upon by group automorphisms by an arbitrary finite group $\Si$ assuming only that $H^1(\Si;G)=\{1\}$.

\begin{proof}[Proof of Proposition \ref{equivariant_rep}]
The particular group $\piL$ plays no role in what follows and can be replaced by an arbitrary group $\Ga$ on which $\Si$ acts by group automorphisms. The same goes for the condition of the image of the sub-group $\Z$, which can be omitted on both sides. Given a $\Si$-equivariant group homomorphism $\rho:\Ga\lra G$, it is straightforward to check that the map $\widehat{\rho}:(\ga,\si)\lmt(\rho(\ga),\si)$ is a homomorphism of $\Si$-augmentations from $\Ga\rtimes\Si$ to $G\rtimes\Si$, satisfying in particular $\widehat{\rho}(1,\si)=(1,\si)$, and that this induces a map \begin{equation}\label{equivariant_rep_bijection}\begin{array}{ccc} \Hom(\Ga;G)^\Si / G^\Si & \lra & \Hom_\Si(\Ga\rtimes\Si;G\rtimes\Si) / G\\ {[}\rho{]} & \lmt & {[}\widehat{\rho}{]}\end{array}.\end{equation} To show that the map \eqref{equivariant_rep_bijection} is bijective when $H^1(\Si;G)=\{1\}$, let us consider an arbitrary homomorphism of $\Si$-augmentations $\chi:\Ga\rtimes\Si\lra G\rtimes\Si$  and the group homomorphism $\rho=\chi|_{\Ga}:\Ga\lra G$. Let us then define, for all $\si\in\Si$ the element $a_\si\in G$ by the condition $\chi(1,\si)=(a_\si,\si)$. The fact that $\chi$ is a group homomorphism implies that $(a_\si)_{\si\in\Si}$ is a normalized $1$-cocycle with values in $G$: $$(a_{\si\si'},\si\si') = \chi(1,\si\si') = \chi(1,\si)\chi(1,\si') = (a_\si \si(a_{\si'}),\si\si').$$ Let us then compute $\chi(\ga,\si)$ in two different ways. On the one hand, $\chi(\ga,\si)= \chi(\ga,1)\chi(1,\si) = (\rho(\ga)a_\si,\si)$ and on the other hand, $\chi(\ga,\si) = \chi(1,\si)\chi(\si^{-1}(\ga),1) = (a_\si(\si\rho\si^{-1})(\ga),\si)$. So, for all $\si\in\Si$, one has \begin{equation}\label{almost_equiv} (\Int_{a_\si}\circ\si)\circ\rho\circ\si^{-1} = \rho.\end{equation} Since $H^1(\Si;G)$ is assumed to be trivial, there exists $h\in G$ such that, for all $\si\in\Si$, $a_\si=h^{-1}\si(h)$ so Equation \eqref{almost_equiv} becomes $\si\circ(\Int_h\circ\rho)\circ\si^{-1} = \Int_h\circ\rho$, which precisely means that $(\Int_h\circ\rho):\Ga\lra G$ is $\Si$-equivariant. Moreover, for all $(\ga,\si)\in\Ga\times\Si$, $\widehat{(\Int_h\circ\rho)}(\ga,\si) = \big( (\Int_h\circ\rho)(\ga),\si\big) = \big(h\rho(\ga)h^{-1},\si\big)$ and 
\begin{eqnarray*}
\big(\Int_{(h,1)}\circ\chi\big) (\ga,\si) & = & (h,1)\chi(\ga,1)\chi(1,\si)(h,1)^{-1} \\
 & = & (h,1) (\rho(\ga),1)(h^{-1}\si(h),\si)(h^{-1},1) \\
 & = & (h\rho(\ga)h^{-1},\si)
\end{eqnarray*}
so $\widehat{(\Int_h\circ\rho)} = \Int_{(h,1)}\circ\chi$, which is $G$-conjugate to $\chi$, and we have indeed proved that the map \eqref{equivariant_rep_bijection} is surjective.

To show that it is injective, let us take two $\Si$-equivariant representations $\rho_1,\rho_2:\Ga\lra G$ and assume that $\widehat{\rho_2}=\Int_{(g,1)}\circ\widehat{\rho_1}$ for some $g\in G$. Then, for all $\si\in \Si$, $$(1,\si)=\widehat{\rho_2}(1,\si) = \Int_{(g,1)}\big(\widehat{\rho_1}(1,\si)\big) = (g,1)(1,\si)(g^{-1},1)=(g\si(g^{-1}),1)$$ so $\si(g)=g$ and $\rho_1$ and $\rho_2$ are in fact $G^{\Si}$-conjugate, which means that the map \eqref{equivariant_rep_bijection} is injective.
\end{proof}

\noindent Note that the general case, where possibly $H^1(\Si;G)\neq\{1\}$, is not much more difficult to handle within the same framework: all homomorphisms of $\Si$-augmentations $\chi:\Ga\rtimes\Si\lra G\rtimes \Si$ still come from $\Si$-equivariant representations $\rho:\Ga\lra G$, provided one adds new actions of $\Si$ on $G$ indexed by elements of $H^1(\Si;G)$ and differing from the original action of $\Si$ on $G$ only by an inner automorphism. Indeed, there is a homeomorphism \begin{equation}\label{general_case_equivariant_rep}
\begin{array}{ccc}
\displaystyle \bigsqcup_{[a]\in H^1(\Si;G)} \Hom(\Ga;G)^{\Si_a} / G^{\Si_a} & \overset{\simeq}{\lra} & \Hom_\Si(\Ga\rtimes\Si;G\rtimes\Si) / G \\
 {[}\rho{]} & \lmt & {[}\widehat{\rho}: (\ga,\si) \lmt (\rho(\ga)a_\si,\si){]}
 \end{array}
 \end{equation} generalizing \eqref{equivariant_rep_bijection}, where the notation $\Si_a$ means that $\Si$ acts on $G$ via $\si\cdot g := \Int_{a_\si}\,\si(g)$. This is an action because $a$ is a $1$-cocycle and it induces a new action of $\Si$ on $\Hom(\Ga;G)$, defined by $\si\cdot\rho = (\Int_{a_\si}\circ\si)\circ\rho\circ\si^{-1}$, which is compatible with the previous $\Si$-action on $G$ and the $G$-conjugacy action on $\Hom(\Ga;G)$ in the sense that $\si\cdot(g\cdot\rho) = (\si\cdot g)\cdot(\si\cdot\rho)$. In particular, $G^{\Si_a}$ indeed acts on $\Hom(\Ga;G)^{\Si_a}$ for all $[a]\in H^1(\Si;G)$, from which one can prove \eqref{general_case_equivariant_rep} exactly in the same way as Proposition \ref{equivariant_rep}.

In view of Section \ref{construction_of_bundles_section}, Proposition \ref{equivariant_rep} means that, when $M^\si\neq\emptyset$, Real and Quaternionic vector bundles can be constructed from $\Si$-equivariant unitary representations of $\piL$. The point is now to show that we can in fact go through this construction directly, without going back to representations of the groups $\piLR$, thus simplifying the proof of Theorem \ref{constr_of_bdles_prop} when $M^\si\neq\emptyset$. First, since an $x\in M^\si$ has been chosen, the universal covering space $\Mt(x)$, which is the set of paths $\delta$ in $M$ such that $\delta(0)=x$, acquires a canonical Real structure $\delta\lmt\si\circ\delta$, which we will simply denote by $\si(\delta)$. Second, recall that $\pi_1(S(L);\ov{x})$ acts on $\Mt(x)$ via the map $\pi_1(S(L);\ov{x})\lra\piCx\simeq\Aut(\Mt(x)/M)$. Then, given a $\Si$-equivariant representation $\rho:\pi_1(S(L);\ov{x})\lra\U(r)$, we can endow the product bundle $\Mt(x)\times\C^r$ with the Real (resp.\ Quaternionic) structure 
\begin{equation}\label{from_equiv_rep_to_bdles}
\widetilde{\tau}: \begin{array}{ccc} \Mt(x) \times \C^r & \lra & \Mt(x)\times\C^r \\ \big(\delta,v\big) & \lmt & \big(\si(\delta), \si(v)\big) \end{array}\ \mathrm{where}\ \si(v)= \left\{ \begin{array}{cl}
\ov{v} & \mathrm{if}\ \si=\si_\R,\\
J\ov{v} & \mathrm{if}\ \si=\si_\H,
\end{array}\right.
\end{equation} and this Real (resp.\ Quaternionic) structure $\widetilde{\tau}$ is immediately seen to induce a Real (resp.\ Quaternionic) structure $\tau$ on $\cE_\rho=\Mt(x)\times_{\rho}\C^r$. Evidently, $\U(r)^\Si$-conjugate $\Si$-equivariant representations $\rho$ give rise to isomorphic Real (resp.\ Quaternionic) bundles. By Remark \ref{generalized_orbifold_rep_space}, if $L=L_0=M\times\C$, then Proposition \ref{equivariant_rep} says that, when $M^\si\neq\emptyset$, $\Hom_{\Si}(\piR;\U(r)\rtimes\Si) / \U(r) \simeq \Hom(\piC;\U(r))^\Si / \U(r)^\Si.$

\section{Moduli of Real and Quaternionic holomorphic vector bundles}\label{moduli_of_real_and_quat_bundles}

\subsection{Stability}\label{stability_Section}

Let $\cE$ be a holomorphic vector bundle over a Klein surface $(M,\si)$. We denote by $\si(\cE)$ the holomorphic vector bundle $\os{\cE}$, explicitly defined as follows. Let $(U_i)_{i\in I}$ be a covering of $M$ by open sets which are trivializing for $\cE$. By considering the trivializing open sets $U_{\si(i)}:=\si(U_i)$ and enlarging $I$ if necessary, we may assume that $\Si$ acts on $I$. If $\cE$ is represented by the holomorphic $1$-cocycle $g_{ij}: U_i\cap U_j \lra \GL(r;\C)$ then $\cEs$ is represented by the holomorphic $1$-cocycle $\si(g_{ij}):=\ov{g_{ij}\circ\si^{-1}}$. Since $\deg(\cEs)=\deg(\cE)$, this defines a left action of $\Si=\Gal(\C/\R)\simeq \{1;\si\}$ on the set of isomorphism classes of holomorphic vector bundles of rank $r$ and degree $d$ on $M$. Moreover, if $\phi:\cE_1\lra \cE_2$ is a homomorphism of holomorphic vector bundles, then there is a well-defined homomorphism of holomorphic vector bundles $\si(\phi):\si(\cE_1) \lra \si(\cE_2)$. By definition of $\cEs=\os{\cE}$, there is always an anti-holomorphic map $\si_{\cE}: \cE \lra \cEs$ such that the diagram

\begin{equation}\label{pullback_map} 
\begin{CD}
\cE @>{\si_{\cE}}>> \cEs \\ 
@VVV @VVV \\
M @>{\si}>> M
\end{CD}
 \end{equation}

\noindent is commutative and $\si_{\cE}$ is fiberwise $\C$-anti-linear. If there exists a ($\C$-linear) isomorphism of holomorphic vector bundles $\phi_\si:\cEs\lra E$ (over $\Id_M$ on the base),


\noindent in which case we say that $\cE$ is self-conjugate (for instance, Real and Quaternionic bundles are self-conjugate and so is the direct sum of a Real and a Quaternionic bundle), then the anti-holomorphic map \begin{equation}\label{real_or_quat_structure_on_self_conjugate_bundles} \tau:= \phis \si_{\cE}\end{equation} satisfies conditions (1) and (2) of Definition \ref{def_real_and_quat_bundles}. The pair $(\cE,\tau)$ thus defined is a Real (resp.\ Quaternionic) vector bundle if and only if $\si(\phis)=\phis^{-1}$ (resp.\ $\si(\phis)=-\phis^{-1}$), as follows from the commutativity of the diagram

$$
\xymatrix{
\cE \ar[d] \ar[r]^{\si_{\cE}} & \cEs \ar[d] \ar[r]^{\phi_\si} & \cE \ar[ld] \\
M \ar[r]^{\si} & M &
}
$$

\noindent and from the fact that $\tau^2= (\phi_\si \si_\cE) (\phi_\si \si_\cE) = \phi_\si \si(\phi_\si)$.

The slope of a non-zero complex vector bundle $\cE$ on a compact connected Riemann surface $M$ is the ratio $\mu(\cE)=\frac{\deg(\cE)}{\rk(\cE)}\in\mathbb{Q}$ of its degree by its rank. Let us recall the various notions of slope stability for Real and Quaternionic vector bundles over Klein surfaces (\cite{Wang_NS,Sch_JSG}).

\begin{definition}[Stability conditions for Real and Quaternionic vector bundles]\label{def_stability_cond}
Let $(\cE,\tau)$ be a non-zero Real (resp.\ Quaternionic) vector bundle on $(M,\sigma)$. We call a sub-bundle of $\cE$ non-trivial if it is distinct from $0$ and $\cE$. Then $(\cE,\tau)$ is said to be: 
\begin{enumerate}
\item stable if, for all non-trivial, $\tau$-invariant sub-bundles $\cF \subset \cE$, the slope stability condition $\mu(\cF) < \mu(\cE)$ is satisfied.
\item semi-stable if, for all non-trivial, $\tau$-invariant sub-bundles $\cF \subset \cE$, one has $\mu(\cF) \leq \mu(\cE).$
\item geometrically stable if the underlying holomorphic bundle $\cE$ is stable, that is, if, for all non-trivial sub-bundles $\cF \subset \cE$,  one has $\mu(\cF) < \mu(\cE).$
\item geometrically semi-stable, if the underlying holomorphic bundle $\cE$ is semi-stable, that is, if for all non-trivial sub-bundles $\cF \subset \cE$,  one has $\mu(\cF) \leq \mu(\cE).$
\end{enumerate}
\end{definition}

\noindent Condition (1) for Real vector bundles was first studied by Wang in \cite{Wang_NS}. Evidently, (3) implies (1) and (4) implies (2). Given a non geometrically semi-stable vector bundle $\cE$, the existence of a unique sub-bundle of maximal rank among sub-bundles of maximal slope (this sub-bundle is called the destabilizing bundle of $\cE$) shows that, actually, (2) implies (4). Indeed, the destabilizing bundle $\cF$ of a Real (resp.\ Quaternionic) bundle $(\cE,\tau)$ satisfies $\phis^{-1}(\si(\cF))=\cF$ by uniqueness of the destabilizing bundle, so $\cF$ is necessarily $\tau$-invariant, contradicting semi-stability in the Real (resp.\ Quaternionic) sense. Thus, being semi-stable as a Real (resp.\ Quaternionic) bundle is equivalent to being geometrically semi-stable and Real (resp.\ Quaternionic). Being stable as a Real (resp.\ Quaternionic) bundle, in contrast, is not equivalent to being both geometrically stable and Real (resp.\ Quaternionic). We can, nonetheless, completely characterize all stable Real (resp.\ Quaternionic) vector bundles.

\begin{proposition}[\cite{Sch_JSG}]\label{charac_of_stability}
Let $(\cE,\tau)$ be a stable Real (resp.\ Quaternionic) vector bundle. Then either $\cE$ is geometrically stable or there exists a stable holomorphic vector bundle $\cF$ such that $$(\cE,\tau) \simeq \left( \cF\oplus \si(\cF),\ \tau_+:=\begin{pmatrix} 0 & \si_{\cF}^{-1} \\ \si_{\cF} & 0 \end{pmatrix} \right),$$ resp.\ $$(\cE,\tau) \simeq \left( \cF\oplus \si(\cF), \ \tau_-:=\begin{pmatrix} 0 & -\si_{\cF}^{-1} \\ \si_{\cF} & 0 \end{pmatrix} \right),$$ as a Real (resp.\ Quaternionic) vector bundle, where $\si_{\cF}:\cF \lra \si(\cF)$ is the map defined in \eqref{pullback_map}. Moreover, in the Real case one has $\si(\cF)\not\simeq\cF$.\\ If $(E,\tau)$ is geometrically stable then $\End(\cE)\simeq\R$ and otherwise $$\End(\cE)\simeq \left\{ \begin{pmatrix} \la & \\ & \ov{\la} \end{pmatrix} : \la\in\C \right\} \simeq \C$$ as $\R$-algebras.
\end{proposition}

\noindent So, there are stable Real (resp.\ Quaternionic) bundles which are not geometrically stable (they are, however, poly-stable in the complex sense). This is actually a good thing because it gives enough stable objects for the following result to hold.

\begin{theorem}[\cite{Sch_JSG}]\label{JH_filt}
Let $(\cE,\tau)$ be a semi-stable Real (resp.\ Quaternionic) vector bundle. Then there exists a filtration $0=\cF_0 \subset \cF_1 \subset \ldots \subset \cF_k=\cE$ by $\tau$-invariant holomorphic sub-bundles such that:
\begin{enumerate}
\item for all $i\in\{1;\,\ldots;k\}$, the Real (resp.\ Quaternionic) bundle $(\cF_i/\cF_{i-1},\tau_i)$ is stable with respect to the Real (resp.\ Quaternionic) structure $\tau_i$ induced by $\tau$,
\item for all $i\in\{1;\,\ldots;k\}$, $\mu(\cF_i/\cF_{i-1}) = \mu(\cE).$
\end{enumerate}
\noindent Moreover, the graded isomorphism class of the associated graded Real (resp.\ Quaternionic) object $\left(\bigoplus_{i=1}^k \cF_i/\cF_{i-1}, \ \bigoplus_{i=1}^k \tau_i\right)$ is independent of the choice of the filtration. We denote it by $\gr(\cE,\tau)$.
\end{theorem}

\noindent A filtration of $(\cE,\tau)$ satisfying the conditions of Theorem \ref{JH_filt} is called a Real (resp.\ Quaternionic) Jordan-H\"older filtration. It is not, in general, a Jordan-H\"older filtration of the holomorphic vector bundle $\cE$: for instance, $(\cF\oplus\si(\cF),\tau_\pm)$ is stable in the Real (resp.\ Quaternionic) sense, but its holomorphic Jordan-H\"older filtration have length $2$ (this example also shows that it is not true in general that a semi-stable Real (resp.\ Quaternionic) vector bundle admits a filtration by $\tau$-invariant sub-bundles the successive quotients of which are geometrically stable). The following definitions are inspired from the classical holomorphic case (\cite{Seshadri}).

\begin{definition}[$S$-equivalence, \cite{Sch_JSG}]
Two semi-stable Real (resp.\ Quaternionic) vector bundles $(\cE_1,\tau_1)$ and $(\cE_2,\tau_2)$ are called $S$-equivalent in the Real (resp.\ Quaternionic) sense if $\gr(\cE_1,\tau_1) = \gr(\cE_2,\tau_2)$, i.e.\ if the graded bundles associated to any choice of Real (resp.\ Quaternionic) Jordan-H\"older filtrations are isomorphic as graded Real (resp.\ Quaternionic) vector bundles.
\end{definition}

\begin{definition}[Poly-stable Real and Quaternionic vector bundles, \cite{Sch_JSG}]\label{ps_objects}
Let $(\cE,\tau)$ be a Real (resp.\ Quaternionic) vector bundle. Then $(\cE,\tau)$ is called poly-stable if it is isomorphic, as a Real (resp.\ Quaternionic) vector bundle, to a direct sum $(\cE_1\oplus \ldots \oplus \cE_k,\tau_1\oplus\ldots \oplus \tau_k)$ of stable Real (resp.\ Quaternionic) vector bundles satisfying $\mu(\cE_1)=\mu(\cE_2)= \ldots = \mu(\cE_k).$ This implies that $\cE$ is semi-stable and that $\mu(\cE)=\mu(\cE_i)$ for all $i\in\{1;\,\ldots;k\}$.
\end{definition}

\noindent For instance, any graded object in the Real (resp.\ Quaternionic) $S$-equivalence class of a semi-stable Real (resp.\ Quaternionic) bundle is poly-stable in the sense of Definition \ref{ps_objects}. The next result follows almost directly from Proposition \ref{charac_of_stability}.

\begin{proposition}[\cite{Sch_JSG}]\label{ps_and_tau_comp}
Let $(\cE,\tau)$ be a Real (resp.\ Quaternionic) vector bundle. Then $(\cE,\tau)$ is poly-stable in the Real (resp.\ Quaternionic) sense if and only if it is poly-stable in the complex sense. In other words, a poly-stable Real (resp.\ Quaternionic) vector bundle is a holomorphic vector bundle which is both poly-stable and Real (resp.\ Quaternionic).
\end{proposition}

\begin{proof}
By Proposition \ref{charac_of_stability}, a poly-stable Real (resp.\ Quaternionic) vector bundle is a direct sum of holomorphic vector bundles of the same slope which are stable in the complex sense, so it is poly-stable in the complex sense. Conversely, if $(\cE,\tau)$ is a Real (resp.\ Quaternionic) vector bundle which is poly-stable in the complex sense, then $\cE \simeq \cE_1\oplus \ldots \oplus \cE_k$ as holomorphic vector bundles, with each $\cE_i$ stable (in the complex sense) and of slope $\mu(\cE)$. If $k=1$, then $(\cE,\tau)$ is geometrically stable. If $k=2$, then $\cE\simeq\si(\cE)$ implies that either $\si(\cE_i)\simeq \cE_i$ for $i=1,2$, or $\si(\cE_1)\simeq \cE_2$. In both cases, $(\cE,\tau)$ is poly-stable (in fact, here, stable) as a Real (resp.\ Quaternionic) vector bundle. The general case follows by induction on $k$.
\end{proof}

In analogy with the construction of Seshadri (\cite{Seshadri}), Theorem \ref{JH_filt} makes it natural to introduce the following moduli sets.

\begin{definition}\label{def_moduli_spaces_Real_Quat_bdles}
We denote by $\ModRd$ the set of Real $S$-equivalence classes of semi-stable Real vector bundles of rank $r\geq 1$ and degree $d\in\Z$ and by $\ModRdw$ the subset of $\ModRd$ consisting of those Real vector bundles $(\cE,\tau)$ satisfying $w_1(\cE^\tau) = w=(s_1,\ldots,s_n)\in(\Z/2\Z)^n$, so $$\ModRd \quad = \bigsqcup_{|\vw|=d\,\mod{2}} \ModRdw,$$ where $|\vw|=s_1+\ldots+s_n$. Likewise, we denote by $\ModHd$ the set of Quaternionic $S$-equivalence classes of semi-stable Quaternionic vector bundles of rank $r\geq 1$ and degree $d\in\Z$.
\end{definition}

We take a moment here to make the following observation: it is not true in general that $\ModRd$ is the set of real points of the variety $\ModCd$ with respect to the Galois action induced by $\cE\lmt\si(\cE)$ (this action indeed takes a semi-stable bundle to a semi-stable one and sends a Jordan-H\"older filtration of $\cE$ to a Jordan-H\"older filtration of $\cEs$). For instance it is not true when $r=2$ and $d=0$, which is of course a case of interest. The point is that the fixed points of $\Si$ in $\ModCd$ do not necessarily come from Real vector bundles: they may come from Quaternionic vector bundles and this phenomenon is exactly due to the presence of non-trivial automorphisms for the objects parametrized by the moduli space $\ModCd$. Also, two Real (resp.\ Quaternionic) vector bundles whose underlying holomorphic vector bundles are isomorphic could be non-isomorphic as Real (resp.\ Quaternionic) vector bundles (this last phenomenon does in fact not occur in the geometrically stable case; see Proposition \ref{moduli_of_geom_stable_Real_and_quat}). To further formalize this, let us restrict our attention to the smooth dense part $\ModCds$ of $\ModCd$ consisting of isomorphism classes of stable holomorphic vector bundles of rank $r$ and degree $d$ (\cite{Mumford_ICM,Seshadri}). A $\Si$-invariant class therein is represented by a self-conjugate stable holomorphic vector bundle $\cE$ on $M$, meaning that there exists a ($\C$-linear) isomorphism of holomorphic vector bundles $\phi_\si:\cEs\overset{\simeq}{\lra}\cE$ (covering the identity map on $M$). Since $\cE$ is stable, it is in particular simple, in the sense that $\Aut(\cE)\simeq \C^*$. Given $\lambda\in\Si=\Gal(\C/\R)\simeq \Z/2\Z$, we set $\phi_\lambda=\Id_\cE$ if $\lambda$ is trivial and $\phi_\lambda=\phis$ if $\lambda$ is the non-trivial element of $\Si$. Then, given $(\lambda,\lambda')\in\Si\times \Si$, let us define the automorphism $u(\lambda,\lambda')$ of $\cE$ by requiring that the diagram

\begin{equation}\label{def_cocycle_diag}
\xymatrix{
(\lambda\lambda')(\cE) \ar[d]^{\phi_{\lambda\lambda'}} \ar[r]^{\lambda(\phi_{\lambda'})} & \lambda(\cE) \ar[r]^{\phi_\lambda} & \cE \\
\cE \ar@{-->}[rru]_{u(\lambda,\lambda')} & &
}
\end{equation}
be commutative. Explicitly, $u(\lambda,\lambda'):=\phi_\lambda \lambda(\phi_{\lambda'}) \phi_{\lambda\lambda'}^{-1}$.

\begin{proposition}\label{self_conj_simple_bundle}
Given a self-conjugate simple bundle $\cE$, the map $u: \Si\times \Si \lra \C^*$ defined in \eqref{def_cocycle_diag} is a $2$-cocycle in Galois cohomology. Consequently, it defines a class $$[u]\in H^2(\Si;\C^*) \simeq \Br(\R)$$ in the Brauer group of $\R$ and this defines a map, that we may call the type map, $$\cT: \ModCds^\Si \lra \Br(\R)=\{\R;\H\}\simeq\Z/2\Z$$ from the Galois-invariant part of $\ModCds$ to the Brauer group of $\R$. If $\cE\in\cT^{-1}(\{\R\})$, then $\cE$ admits a Real structure $\tau$ and, if $\cE\in\cT^{-1}(\{\H\})$, then $\cE$ admits a Quaternionic structure. In particular, a geometrically stable self-conjugate vector bundle is either Real or Quaternionic and cannot be both.
\end{proposition}

\begin{proof}
Recall that $u(\lambda,\lambda')= \phi_\lambda \lambda(\phi_{\lambda'}) \phi_{\lambda\lambda'}^{-1}$, where $\phi_\lambda$ is, for all $\lambda\in\Si$, an isomorphism between $\lambda(\cE)$ and $\cE$. That $u$ is a $2$-cocycle follows from a simple computation and the associated cohomology class $[u]$ does not depend on the choice of the family $(\phi_\lambda)_{\lambda\in\Si}$. If $[u]=1$ in $H^2(\Si;\C^*)\simeq\Br(\R)\simeq\{\pm 1\}$, then $u$ is a coboundary: $u(\lambda,\lambda') = a(\lambda\lambda') a(\lambda)(\lambda\cdot a(\lambda'))^{-1}$, where $a:\Si\lra \Aut(\cE)$ and $\lambda\cdot a(\lambda') = \phi_{\lambda}\lambda(a(\lambda')) \phi_{\lambda}^{-1}$ (this does not depend on the choice of the isomorphism $\phi_{\lambda}:\lambda(\cE)\lra\cE$). Setting $\psi_{\lambda}:=a(\lambda)\phi_{\lambda}$ yields $\psi_{\lambda}\lambda(\psi_{\lambda'}) = \psi_{\lambda\lambda'}$. In particular, if $\lambda=\lambda'=\si$, then $\si(\psi_{\si})=\psi_\si^{-1}$. Likewise, $[u]=-1$ yields $\psi_{\si}^{\si}=-\psi_\si^{-1}$. If we set $\tau:=\psi_\si\si_{\cE}$, where $\si_{\cE}$ is defined as in \eqref{pullback_map}, then $\tau^2 = \psi_\si\si_{\cE} \psi_\si\si_{\cE} = \psi_\si \si(\psi_\si) = \pm \Id_{\cE}$.
\end{proof}
The lesson from Proposition \ref{self_conj_simple_bundle} is that there may be $\Gal(\C/\R)$-invariant stable holomorphic vector bundles that do not come from Real vector bundles. This situation contrasts with the one studied by Harder and Narasimhan in \cite{HN}, where all $\Gal(\ov{\mathbb{F}_q}/\mathbb{F}_q)$-invariant bundles are necessarily defined over $\mathbb{F}_q$, reflecting the fact that $\Br(\mathbb{F}_q)$, unlike $\Br(\R)$, is trivial. To conclude the present section, let us identify the fibers of the type map $\cB$ with moduli spaces of Real and Quaternionic vector bundles. We define $\ModRds$, resp. $\ModHds$, to be the set of isomorphism classes of geometrically stable Real, resp. Quaternionic, holomorphic  vector bundles of rank $r$ and degree $d$. By Proposition \ref{self_conj_simple_bundle}, there are surjective maps $$\ModRds\lra\cB^{-1}(\{\R\})\quad\mathrm{and}\quad \ModHds\lra\cB^{-1}(\{\H\}),$$ which in particular endows $\ModRds$ and $\ModHds$ with a natural topology that they inherit from the Hausdorff topology of the complex variety $\ModCds$. In fact, these maps are also injective, as shown by the next result.

\begin{proposition}\label{moduli_of_geom_stable_Real_and_quat}
Two geometrically stable Real (resp. Quaternionic) vector bundles $(\cE_1,\tau_1)$ and $(\cE_2,\tau_2)$ are isomorphic as Real (resp. Quaternionic) vector bundles if and only if they are isomorphic as holomorphic vector bundles. As a consequence, there are bijections $\ModRds\simeq\cB^{-1}(\{\R\})$ and $\ModHds\simeq\cB^{-1}(\{\H\})$.
\end{proposition}

\begin{proof}
It suffices to show that two Real (resp.\ Quaternionic) structures $\tau_1$ and $\tau_2$ on a stable holomorphic vector bundle $\cE$ are conjugate by an automorphism of $\cE$. An elementary proof is available in \cite[Proposition 2.8]{Sch_JSG} but it is interesting to see the group cohomology argument which is hidden there. To that end, let us replace $\tau_1$ and $\tau_2$ by isomorphisms $\phi_\si,\psi_\si:\si(\cE)\lra\cE$ defined as in \eqref{real_or_quat_structure_on_self_conjugate_bundles}. Then $a(\si):=\psi_\si\phi_\si^{-1}$ defines a $\C^*$-valued $1$-cocycle on $\Si$ (even when $\phi_\si$ and $\psi_\si$ both come from Quaternionic structures on $\cE$). Since $H^1(\Si;\C^*)=0$ by Hilbert's 90, we have that $\tau_1$ and $\tau_2$ are indeed conjugate by an automorphism of $\cE$.
\end{proof}

\subsection{Galois-invariant connections}\label{Galois_inv_conn}

We now present the gauge-theoretic construction the moduli spaces $\ModRdw$ and $\ModHd$, which lies at the heart of our proof of the Narasimhan-Seshadri correspondence for Real and Quaternionic vector bundles. The upshot of this construction is that it places Real and Quaternionic vector bundles on an equal footing: both are fixed points of an appropriate Real structure on a space of unitary connections. The general idea behind our approach is the Kempf-Ness theorem (\cite{KN}), or rather the use that Donaldson makes of that theorem in \cite{Don_NS}. In what follows, we fix a compatible Riemannian metric of volume $2\pi$ on the Riemann surface $M$ and we denote by $\vol_M$ the associated volume form. Following Atiyah and Bott in \cite{AB}, we denote by $E$ a Hermitian vector bundle of rank $r$ and degree $d$ on $M$, and $\cG_E$ the group of unitary automorphisms of $E$ (henceforth referred to as the unitary gauge group). Let $\cG_\C$ be the group of all complex linear automorphisms of $E$ (henceforth referred to as the complex gauge group; it can be seen as a complexification $\cG_E^{\,\C}$ of $\cG_E$). Let $\cA_E$ be the (affine) space of unitary connections on $E$. Then the unitary gauge action, defined for $u\in\cG_E$, $A\in\cA_E$ and $s\in\Omega^0(M;E)$ by 
$d_{u\cdot A} (s) = u\big(d_A(u^{-1}s)\big) = (d_A -(d_Au)u^{-1})\, s,$ extends to the complex gauge action defined for $g\in\cG_\C$ by $$d_{g\cdot A}(s) = g\big(\ov{\partial}_A(g^{-1}s)\big) + (g^*)^{-1} \big(\partial_A(g^*s)\big) = \big[ d_A - \big((\ov{\partial}_A g)g^{-1} - \big((\ov{\partial}_A g) g^{-1}\big)^*\big) \big]s,$$ where $g^*$ is the Hermitian adjoint of $g$ and $\ov{\partial}_A$ (resp.\ $\partial_A$) is the $(0,1)$ (resp.\ $(1,0)$) part of $d_A:\Omega^0(M;E)\lra \Omega^1(M;E)=\Omega^{1,0}(M;E) \oplus \Omega^{0,1}(M;E)$ (if $g^*=g^{-1}$, the actions above indeed coincide because $(\ov{\partial}_A g)^* = \partial_A (g^*)$), which puts us in a situation similar to that of the Kempf-Ness theorem, albeit in an infinite-dimensional context. This is interesting because, by the Newlander-Niren\-berg theorem, the $\cG_\C$-orbits in $\cA_E$ are precisely the isomorphism classes of holomorphic structures on $E$. Moreover, Atiyah and Bott have proved (\cite{AB}) that the curvature map $F:\cA_E \lra \Om^2\big(M;\fu(E)\big)\simeq\big(Lie(\cG_E)\big)^*$ (where $\fu(E)$ is the bundle of anti-Hermitian endomorphisms of $E$) is a momentum map for the unitary gauge action on $\cA_E$, and Donaldson has proved (\cite{Don_NS}) that poly-stable $\cG_\C$-orbits (i.e.\ $\cG_\C$-orbits of unitary connections defining a poly-stable holomorphic structure on $E$) are characterized by the fact that they contain a, necessarily unique, $\cG_E$-orbit of unitary connections having constant central curvature, i.e.\ satisfying the momentum map condition \begin{equation}\label{min_YM} F_A = \left(i  \frac{d}{r}\Id_E \right) \vol_M \in \Om^2\big(M;\fu(E)\big).\end{equation} Finally, Daskalopoulos and R{\aa}de have proved that  the closure of a semi\-stable $\cG_\C$-orbit (i.e.\ $\cG_\C$-orbits of unitary connections defining a semi-stable holomorphic structure on $E$) contains a unique poly-stable $\cG_\C$-orbit (\cite{Dask,Rade}). More precisely, if we denote by \begin{equation}\label{YM_def} \cL_{YM}: \begin{array}{ccc}\cA_E & \lra & \R \\ A & \lmt & \int_M \|F_A\|^2\end{array}\end{equation} the Yang-Mills functional of $E$ (the smooth map $\|\cdot\|^2:\Omega^2(M;\mathfrak{u}(E)) \lra \Omega^2(M;\R)$ being induced by the canonical scalar product $(P,Q)\lmt -\mathrm{tr}(PQ)$ on $\mathfrak{u}(r)$), then, given any $A\in\cA_E$, there exists a one-parameter family $(A_t)_{t\geq 0}$ of unitary connections satisfying \begin{equation}\label{YM_flow}\left\{\begin{array}{ccc} A_0 & = & A \\ \frac{d}{dt}A_t & = & -\mathrm{grad}_{A_t} \cL_{YM} \end{array}
\right.\end{equation} and such that $A_{\infty}:= \lim_{t\to +\infty} A_t$ exists and is a critical point of $\cL_{YM}$ (\cite{Dask,Rade}). The connection $A_{\infty}$ is called the limit point of $A$ under the Yang-Mills flow (i.e.\ the negative gradient flow of $\cL_{YM}$) and a critical point of $\cL_{YM}$ is called a Yang-Mills connection. Additionally, the limiting connection $A_{\infty}$ is an absolute minimum of $\cL_{YM}$ if and only if it defines a poly-stable holomorphic structure on $E$, if and only if the original connection $A$ defines a semi-stable holomorphic structure on $E$. And in this case, $\cG_\C\cdot A_{\infty}$ is the only poly-stable orbit contained in the closure $\ov{\cG_\C\cdot A}$ of $\cG_\C\cdot A$ in the semi-stable locus of $\cA_E$ (\cite{Dask,Rade}). In particular, the closed semi-stable $\cG_\C$-orbits are the poly-stable ones. Putting it all together, this shows that an analog, in the present context, of the Kempf-Ness theorem indeed holds: two semi-stable homolorphic structures on $E$ are $S$-equivalent if and only if the closures of the corresponding $\cG_\C$-orbits intersect in the semi-stable locus of $\cA_E$, and the space of closed semi-stable orbits is in bijection with unitary orbits of solutions to the momentum map equation \eqref{min_YM}. Denoting by $\cC_{ss}\subset\cA_E$ the set of unitary connections defining semi-stable holomorphic structures on $E$, we see that the above implies the existence of bijections (in fact, homeomorphisms)

\begin{equation}\label{GIT_pic_cx_case}
\cM^{ss}(E) \simeq \cC_{ss} \quot \cG_\C \simeq \fibre /\, \cG_E\ ,
\end{equation} 

\noindent between

\begin{itemize}
\item the set of $S$-equivalence classes of semi-stable holomorphic structures on $E$,
\item the space of closed semi-stable $\cG_\C$-orbits in $\cA_E$,
\item and the space of unitary gauge orbits of minimal Yang-Mills connections (solutions to Equation \eqref{min_YM}).
\end{itemize} 

\noindent Since any two Hermitian vector bundles of rank $r$ and degree $d$ over a curve are smoothly isometric, $\cM^{ss}(E)$ coincides with the moduli space $\ModCd$.

We now explain the invariant-theoretic picture analogous to \eqref{GIT_pic_cx_case} in the Real and Quaternionic cases (see \eqref{GIT_pic_Real_quat_case}), expanding the results of \cite{Sch_JSG} to include results on the closure of semi-stable orbits of Real and Quaternionic structures. As in \cite{AB}, the first step consists in fixing a Real (resp.\ Quaternionic) Hermitian vector bundle $(E,\tau)$ on $(M,\si)$. This means that the map $\tau:E\lra E$ is a fiberwise $\C$-antilinear isometry covering $\si:M\lra M$ and whose square is equal to $\Id_E$ (resp.\ $-\Id_E$). 

\begin{definition}\label{gauge-theoretic-moduli-spaces}
A holomorphic structure on a Real or Quaternionic Hermitian bundle $(E,\tau)$ is said to be $\tau$-compatible if it turns $\tau$ into an anti-holomorphic map. We denote by $\cM^{ss}(E,\tau)$ the set of $S$-equivalence classes of semi-stable $\tau$-compatible holomorphic structures on $(E,\tau)$.
\end{definition} 

Let us now fix a Real (resp.\ Quaternionic) Hermitian vector bundle $(E,\tau)$ on $(M,\Si)$ of rank $r$ and degree $d$, say, and study the set $\cM^{ss}(E,\tau)$ of $S$-equivalence classes (in the Real (resp.\ Quaternionic) sense) of semi-stable $\tau$-compatible holomorphic structures on $(E,\tau)$. The set of all holomorphic structures on $E$ is the space $\cA_E$ of unitary connections on that bundle. If $(u_{ij})_{(i,j)\in I \times I}$ is a unitary cocycle representing $E$, the cocycle $(\ov{u_{ij}\circ\si^{-1}})_{(\si(i),\si(j))}$ is also unitary, so the bundle $\si(E)$ is naturally Hermitian. Moreover, a unitary connection $A$ on $E$ induces a unitary connection $\si(A)$ on $\si(E)$: if $A$ is locally represented by the anti-Hermitian matrices $(a_i)_{i\in I}$, then $\si(A)$ is locally represented by the matrices $(\ov{a_{i}\circ\si^{-1}})_{\si(i)}$. A similar construction applies to the curvature $F_A$ of the connection $A$ and it is immediate that $F_{\si(A)}=\si(F_A)$. In particular, if $A$ is flat, then $\si(A)$ is also flat and, more generally, if $A$ is projectively flat then $\si(A)$ is also projectively flat, for $$\si(i\frac{d}{r}\vol_M) = \overline{(\si^{-1})^*\left(i\frac{d}{r}\vol_M\right)}=-i\frac{d}{r}((\si^{-1})^*\vol_M)=i\frac{d}{r}\vol_M$$ since $\si$ reverses orientation on $M$. Now, since $\tau^2=\pm\Id_E$, there is a $\C$-linear isomorphism $\phi_\si:\si(E) \lra E$ satisfying $\si(\phi_\si)=\pm\phi_\si^{-1}$ (namely $\phi_\si:=\tau\si_\cE^{-1}$, where $\si_\cE$ is defined as in \eqref{pullback_map}), so we can define the following map \begin{equation}\label{the_involution} \beta: \begin{array}{ccc} \cA_E & \lra & \cA_E \\ A & \lmt & (\phi_\si^{-1})^* \si(A) \end{array}\end{equation} where $(\phi_{\si}^{-1})^*\si(A)$ is the connection on $E$ defined, for all smooth sections $s\in\Omega^0(M;E)$, by $d_{(\phi_{\si}^{-1})^*\si(A)} (s) := \phi_\si\big(d_{\si(A)}(\phi_\si^{-1} s)\big)$.

\begin{proposition}[\cite{Sch_JSG}]\label{tau_inv_conn}
Let $(E,\tau)$ be a Real (resp.\  Quaternionic) Hermitian vector bundle. Then the map $\beta$ defined in \eqref{the_involution} is involutive and the holomorphic structure defined by a unitary connection $A\in\cA_E$ is $\tau$-compatible if and only if $\beta(A)=A.$ We call such a fixed point of $\beta$ a Galois-invariant connection.
\end{proposition}

\begin{proof}
Recall that $\si(\phi_\si)=\pm\phi_\si^{-1}$, depending on whether $\tau$ is Real or Quaternionic. So $\beta^2(A) = (\phi_\si^{-1})^* \si((\phi_\si^{-1})^*\si(A)) = (\si(\phi_\si)^{-1}\phi_\si^{-1})^* \si^2(A) = (\pm\Id_E)^* A = A,$ since $\pm\Id_E$ lies in the center of $\cG_E$ therefore acts trivially on $A$. Moreover, $\beta(A)=A$ if and only if the corresponding covariant derivative $d_A: \Omega^0(M;E) \lra \Omega^1(M;E)$ commutes to the Real (resp.\ Quaternionic) structures of $\Omega^0(M;E)$ and $\Omega^1(M;E)$ induced by $\si$ and $\tau$. This guarantees that $\ker \ov{\partial}_A$ inherits a Real (resp.\ Quaternionic) structure, which in turn endows $E$ with a $\tau$-compatible holomorphic structure (\cite{Sch_JSG}).
\end{proof}

\noindent Note that $\beta$ is involutive even if $\tau$ is Quaternionic. Moreover, the involution $\beta$ is compatible with the Hamiltonian structure of $\cA_E$. Indeed, first note that both the unitary and complex gauge groups of $E$ inherit an involution, also denoted by $\beta$, $$\beta: \begin{array}{ccc} \cG_\C & \lra & \cG_\C \\ g & \lmt & (\phi_\si^{-1})^* \si(g) := \phi_\si \si(g) \phi_\si^{-1} = \tau g\tau^{-1} \end{array}$$ (this indeed preserves the sub-group $\cG_E \subset \cG_\C$ because $\phi_\si$ is an isometry) and that the fixed-point set of $\beta$ consists of gauge transformations of $E$ which commute to $\tau$. Note that $\beta$ induces an involution $R\lmt (\phi_\si^{-1})^* \si(R)$ on $\Omega^2(M;\mathfrak{u}(E)) \simeq (Lie\, \cG_E)^*$. Then, the following result holds.

\begin{theorem}[\cite{Sch_JSG}]\label{Hamil_picture}
The involution $\beta$ defined in \eqref{the_involution} is anti-symplectic with respect to the Atiyah-Bott symplectic form $\omega_A(a,b) = \int_M -\mathrm{tr}(a\wedge b)$ on $\cA_E$. It is an isometry of the K\"ahler metric of $\cA_E$ and it is also compatible with the action of $\cG_\C$ and the momentum map of the induced $\cG_E$-action, in the following sense:
\begin{enumerate}
\item for all $g\in\cG_C$ and all $A\in\cA_E$, $\beta(g\cdot A) = \beta(g)\cdot \beta(A)$,
\item for all $A$ in $\cA_E$, $F_{\beta(A)} = \beta(F_A)$.
\end{enumerate}
\end{theorem}

This implies that $\beta$ induces an anti-holomorphic and anti-symplectic involution $\hat{\beta}$ of the (smooth part of the) quotient $$\fibre /\, \cG_E \simeq \cM^{ss}(E)\simeq \ModCd\, .$$ The involution $\hat{\beta}$ is independent of $\tau$ (in particular, it does not depend on whether $\tau$ is Real or Quaternionic) and it coincides with the canonical Real structure $\cE\lmt \si(\cE)$ of $\ModCd$. In what follows, we will denote by $\cG_E^{\,\tau}$ (resp.\ $\cG_\C^\tau$) the group of fixed points of $\beta$ in $\cG_E$ (resp.\ $\cG_\C$) and call it the $\tau$-unitary (resp.\ $\tau$-complex) gauge group. Let us write $\cA_E^{\, \tau}$ for the set of $\tau$-compatible holomorphic structures on $(E,\tau)$. By Proposition \ref{tau_inv_conn}, $\cA_E^{\, \tau}=\mathrm{Fix}(\beta)$, the set of Galois-invariant unitary connections on $(E,\tau)$ and it is a closed affine subspace of $\cA_E$. Property (1) of Theorem \ref{Hamil_picture} shows that the group $\cG_\C^\tau$ and its sub-group $\cG_E^{\,\tau}$ act on $\cA_E^{\, \tau}$. Moreover, as seen in \cite{Sch_JSG}, $\cG_\C^\tau$-orbits in $\cA_E^{\, \tau}$ are precisely isomorphism classes of Real (resp.\ Quaternionic) vector bundles that are smoothly isomorphic to $(E,\tau)$. We then have the following result, which is an analog of the theorems of Donaldson and Daskalopoulos-R{\aa}de and which will be proved in Section \ref{section_YM_flow}.

\begin{theorem}\label{main_step_towards_NS}
Let $(E,\tau)$ be a Real (resp.\ Quaternionic) Hermitian vector bundle on a compact connected Real Riemann surface $(M,\Si)$ of genus $g\geq 2$. Let $\cA_E^{\, \tau}$ be the space of Galois-invariant unitary connections on $(E,\tau)$ and let $A\in\cA_E^{\, \tau}$ define a semi-stable $\tau$-compatible holomorphic structure on $(E,\tau)$. Denote by $A_{\infty}$ the limit point of $A$ under the Yang-Mills flow \eqref{YM_flow}. Then:
\begin{enumerate}
\item $A_{\infty}\in\cA_E^{\, \tau}$ and it defines a poly-stable $\tau$-compatible holomorphic structure on $(E,\tau)$,
\item $\cG_\C^\tau\cdot A_{\infty}$ is the unique poly-stable $\cG_\C^\tau$-orbit contained in the closure $\ov{\cG_\C^\tau\cdot A}$ of $\cG_\C^\tau\cdot A$ in the semi-stable locus of $\cA_E^{\,\tau}$. In particular, the closed semi-stable $\cG_\C^\tau$-orbits in $\cA_E^{\, \tau}$ are the poly-stable ones,
\item the space of closed semi-stable $\cG_\C^\tau$-orbits in $\cA_E^{\, \tau}$ is in bijection with the Lagrangian quotient $$\left( \fibre \cap \cA_E^{\, \tau}\right) \big/\, \cG_E^{\,\tau}$$ i.e.\ with the space of $\cG_E^{\,\tau}$-orbits of Galois-invariant minimal Yang-Mills connections on $E$.
\end{enumerate}
\end{theorem}

\noindent In other words, two semi-stable $\tau$-compatible holomorphic structures on $(E,\tau)$ are $S$-equivalent in the Real (resp.\ Quaternionic) sense if and only if the closures (in the semi-stable locus of $\cA_E^{\, \tau}$) of the corresponding $\cG_\C^\tau$-orbits intersect. Since $\cC_{ss}$ is a $\beta$-invariant and $\cG_{\C}$-invariant subset of $\cA_E$, we have in particular that $\cG_{\C}^{\tau}$ acts on $\cC_{ss}^{\,\tau}$ and, analogously to \eqref{GIT_pic_cx_case}, there are homeomorphisms 

\begin{equation}\label{GIT_pic_Real_quat_case}
\cM^{ss}(E,\tau) \simeq \cC_{ss}^{\, \tau} \quot \cG_\C^\tau \simeq \left( \fibre \cap \cA_E^{\, \tau}\right) /\, \cG_E^{\,\tau}\ ,
\end{equation} 
 
\noindent between 

\begin{itemize}
\item the set of $S$-equivalence classes of semi-stable $\tau$-compatible holomorphic structures on $(E,\tau)$,
\item the space of closed semi-stable $\cG_\C^\tau$-orbits in $\cA_E^{\, \tau}$,
\item and the space of $\tau$-unitary gauge orbits of Galois-invariant minimal Yang-Mills connections (invariant solutions to Eq. \eqref{min_YM}).
\end{itemize} 

\noindent The proof of Theorem \ref{main_step_towards_NS} will be given in Section \ref{section_YM_flow}. For now, we use it to give $\cM^{ss}(E,\tau)$ a topology.

\begin{definition}\label{define_a_topology}
The set $\fibre\cap \cA_E^{\,\tau}$ inherits the subspace topology from the space $\cA_E$ of unitary connections on $E$, as defined in \cite{AB}. The groups $\cG_E^{\,\tau}\subset\cG_\C^{\tau}$ inherit the sub-space topology from the groups $\cG_E\subset\cG_\C$. The set $\cM^{ss}(E,\tau)\simeq(\fibre\cap \cA_E^{\,\tau})/\cG_E^{\,\tau}$ is endowed with the quotient topology.
\end{definition}

\noindent We recall from \cite{AB,Don_NS} that the space $\cA_E$ is the space of unitary connections on $E$ of Sobolev class $L^2_1$. This Sobolev norm turns $\cA_E$ into a Banach affine space. And $\cG_E$ is the set of unitary gauge transformations of Sobolev class $L^2_2$ and is a Banach Lie group. This is compatible with the study of holomorphic structures on $E$ because gauge transformations of Sobolev class $L^2_2$ preserve the topology of the bundle and because any complex gauge orbit of unitary connections of Sobolev class $L^2_1$ contains a smooth unitary connection (\cite[Lemma 14.8]{AB}). Moreover, if two smooth unitary connections are related by a gauge transformation then the latter is necessarily smooth (\cite[Lemma 14.9]{AB}). Arguably, the definition of the topology of $\cM^{ss}(E,\tau)$ given above can seem \textit{ad hoc} for our purposes of establishing a homeomorphism between gauge equivalence classes of Galois-invariant connections and conjugacy classes of Galois-equivariant unitary representations but it is in fact natural. Indeed, it was shown in \cite{Sch_JSG} that the  Lagrangian quotients $(\fibre\cap\cA_E^{\,\tau})/\cG_E^{\,\tau}$ embed continuously into the set of Real points of $(\ModCd,\Si)$, where $r=\rk(E)$ and $d=\deg(E)$, so the topology introduced in Definition \ref{define_a_topology} makes $\cM^{ss}(E,\tau)$ a topological subspace of $\ModCd^\Si\subset \ModCd$ when $\ModCd$ is equipped with its Hausdorff topology of complex variety (in which it is homeomorphic to the symplectic quotient $\fibre/\,\cG_E$, by the Narasimhan-Seshadri theorem). We note that, when $r\wedge d=1$, $\cM^{ss}(E,\tau)$ is actually a connected component of $\ModCd^\Si$ in said topology (\cite{Sch_JSG}).

\subsection{Properties of the Yang-Mills flow on Galois-invariant connections}\label{section_YM_flow}

The first step to prove Theorem \ref{main_step_towards_NS} is to show that the space $\cA_E^{\,\tau}$ of Galois-invariant unitary connections on $(E,\tau)$ is invariant under the Yang-Mills flow. Recall that in \eqref{YM_def} we denoted by $\cL_{YM}$ the Yang-Mills functional and that we have defined in \eqref{the_involution} an involutive isometry $\beta: \cA_E \lra \cA_E$ such that, by Theorem \ref{Hamil_picture}, one has, for all $A\in \cA_E$, $F_{\beta(A)} = \beta(F_A)$.

\begin{lemma}\label{inv_under_flow}
The map $\beta$ satisfies $\cL_{YM}\circ\beta=\cL_{YM}$ and, for all $A\in\cA_E$, $$\mathrm{grad}_{\beta(A)}\cL_{YM} = T_A\beta\cdot \mathrm{grad}_A\cL_{YM}.$$ In particular, if $A\in\cA_E^{\,\tau}=\mathrm{Fix}(\beta)$, then $\mathrm{grad}_A\cL_{YM} \in \mathrm{Fix}(T_A\beta) = T_A(\cA_E^{\,\tau})$.
\end{lemma}

\begin{proof}
One has $\|F_{\beta(A)}\|^2=\|\beta(F_A)\|^2 = \|F_A\|^2$, so $\cL_{YM}\circ\beta=\cL_{YM}$. In particular, for all $A\in\cA_E$, $T_{\beta(A)}\cL_{YM}\circ T_A\beta = T_A\cL_{YM}$. So, by definition of the gradient, one has, for all $v\in T_A(\cA_E)$,
$(\mathrm{grad}_{\beta(A)} \cL_{YM}\,|\, v) 
= T_{\beta(A)}\cL_{YM}\cdot v 
= T_A(\cL_{YM}\circ\beta)\circ (T_A\beta)^{-1} \cdot v 
= (\mathrm{grad}_A\cL_{YM}\, |\, (T_A\beta)^{-1}\cdot v).$
Since $\beta$ is an isometry of $\cA_E$, this shows that $\mathrm{grad}_{\beta(A)}\cL_{YM}=T_A\beta\cdot\mathrm{grad}_A\cL_{YM}$ and the rest follows. 
\end{proof}

\noindent We can now prove Part (1) of Theorem \ref{main_step_towards_NS}.

\begin{proof}[Proof of Part \emph{(1)} of Theorem \ref{main_step_towards_NS}]
As a consequence of Lemma \ref{inv_under_flow}, $\cA_E^{\,\tau}$ is invariant under the Yang-Mills flow \eqref{YM_flow}. Since $\cA_E^{\,\tau}=\Fix(\beta)$ is closed in $\cA_E$, the limiting connection $A_{\infty}$ of a Galois-invariant connection $A$ is also Galois-invariant. Moreover, if $A$ is semi-stable, then $A_{\infty}$ is of constant central curvature by the Daskalopoulos-R{\aa}de theorem. Since a connection which is both of constant central curvature and Galois-invariant defines a holomorphic structure on $(E,\tau)$ which is both poly-stable (by Donaldson's theorem) and $\tau$-compatible (by Proposition \ref{tau_inv_conn}), Part (1) of Theorem \ref{main_step_towards_NS} is proved (note that, by Proposition \ref{ps_and_tau_comp}, being both poly-stable and $\tau$-compatible is indeed the same as being a poly-stable $\tau$-compatible holomorphic structure).
\end{proof}

In order to prove Parts (2) and (3) of Theorem \ref{main_step_towards_NS}, we show that the Yang-Mills flow moves a Galois-invariant connection inside its $\cG_{\C}^{\tau}$-orbit, just as it moves a general unitary connection inside its $\cG_\C$-orbit (\cite{AB}), and with that we will be able to conclude the proof of Theorem \ref{main_step_towards_NS}.

\begin{lemma}\label{moves_inside_orbit}
Let $A\in\cA_E^{\,\tau}=\Fix(\beta)$ and let $(A_t)_{t\geq 0}$ be the one-parameter family of unitary connections obtained by flowing $A$ along Yang-Mills gradient lines as in Equation \eqref{YM_flow}. Then, for all $t\geq 0$, $A_t\in\cG_{\C}^{\tau}\cdot A$.
\end{lemma}

\begin{proof} If $A\in\cA_E^{\,\tau}$, then Lemma \ref{inv_under_flow} shows that $\mathrm{grad}_{A} \cL_{YM} \in T_A(\cA_E^{\,\tau})$. Moreover, by \cite[(8.12) p.572]{AB}, $\mathrm{grad}_{A} \cL_{YM}\in T_A(\cG_\C\cdot A)$. So $\mathrm{grad}_{A} \cL_{YM} \in T_A(\cA_E^{\,\tau})\cap T_A(\cG_\C\cdot A) = T_A (\cG_\C^{\tau}\cdot A)$, where the equality follows from the fact that if $X\in Lie\,\cG_\C$ and $Y:=\tau(X)-X$ lies in the infinitesimal stabilizer of $A$, then $X':=X+\frac{Y}{2}=\frac{X+\tau(X)}{2}\in Lie\,\cG_\C^{\tau}$ and induces the same infinitesimal action as $X$.
\end{proof}

\begin{proof}[Proof of Parts \emph{(2)} and \emph{(3)} of Theorem \ref{main_step_towards_NS}]
Since $A_{\infty}=\lim_{t\to+\infty} A_t$, one has, by Lemma \ref{moves_inside_orbit}, that $A_{\infty}\in\ov{\cG_\C^\tau\cdot A}$, so $\cG_\C^\tau\cdot A_{\infty}\subset \ov{\cG_\C^\tau\cdot A}$. If $B$ is another Galois-invariant connection with constant central curvature such that $\cG_\C^\tau\cdot B \subset \ov{\cG_\C^\tau\cdot A}$, then by the Daskalopoulos-R{\aa}de theorem, $\cG_\C\cdot A= \cG_\C\cdot B$. But if two minimal Yang-Mills connections lie in the same $\cG_\C$-orbit, then by Donaldson's theorem they lie in the same $\cG_E$-orbit. And if both connections are in addition Galois-invariant, then they lie in the same $\cG_E^{\,\tau}$-orbit (a detailed proof is available in \cite[Prop. 3.6]{Sch_JSG}). Thus, $\cG_{\C}^\tau\cdot A_{\infty}$ is the only poly-stable $\cG_\C^\tau$-orbit contained in $\ov{\cG_{\C}^\tau\cdot A}$ and $\cG_E^{\,\tau}\cdot A_{\infty}$ is the only $\cG_E^{\,\tau}$-orbit of minimal Yang-Mills connections contained in $\cG_{\C}^\tau\cdot A_{\infty}$, which proves Parts (2) and (3) of Theorem \ref{main_step_towards_NS}.
\end{proof}

\noindent As a corollary of Theorem \ref{main_step_towards_NS}, we obtain the following analog of a result of Daskalopoulos and R{\aa}de (\cite{Dask,Rade}). The proof is immediate because the retraction $r$ that we consider here is the restriction of the Daskalopoulos-R{\aa}de retraction and we have proved that the deformation occurs within $\cG_{\C}^{\tau}$-orbits in $\cA_E^{\,\tau}$.

\begin{corollary}\label{deformation_retract}
Recall that we have denoted by $\cC_{ss}^{\,\tau}\subset\cA_E^{\,\tau}$ the set of semi-stable $\tau$-compatible unitary connections on $E$. Then the map $$r:\begin{array}{ccc} \cC_{ss}^{\,\tau} & \lra & \fibre\cap\cA_E^{\,\tau}\\
A & \lmt & A_{\infty}
\end{array}$$ is a $\cG_E^{\,\tau}$-equivariant deformation retraction.
\end{corollary}

\noindent Since polar decomposition in $\cG_{\C}^{\,\tau}$ induces a deformation retraction from $\cG_{\C}^{\tau}$ to $\cG_E^{\,\tau}$, the $\cG_\C^{\tau}$-equivariant cohomology of $\cC_{ss}^{\,\tau}$ coincides with its $\cG_E^{\,\tau}$-equivariant cohomology and Corollary \ref{deformation_retract} implies that $H_{\cG_E^{\,\tau}}^{\,*}(\fibre\cap\cA_E^{\,\tau}) \simeq H_{\cG_{E}^{\,\tau}}^{\,*}(\cC_{ss}^{\,\tau})\simeq H_{\cG_{\C}^\tau}^{\,*}(\cC_{ss}^{\,\tau})$, which is the approach that was implemented in \cite{LS} to compute the equivariant cohomology with mod $2$ coefficients of the moduli spaces $\cM^{ss}(E,\tau)$, whose presentation as a quotient was described in \eqref{GIT_pic_Real_quat_case}.

\section{Invariant connections and parallel transport}\label{inv_conn_and_par_transport}

\subsection{Parallel transport and holonomy representations}\label{parallel_transport}

Let $E$ be a smooth complex vector bundle over a Klein surface $(M,\si)$. Recall that we denote by $\si(E)$ the complex vector bundle $\ov{(\si^{-1})^*E}$. Let $A$ be a linear connection on $E$. Given a path $\ga:[0;1]\lra M$, we denote by $T_{\,\ga}^A:E_{\ga(0)} \lra E_{\ga(1)}$ the parallel transport operator along $\ga$ with respect to the connection $A$: if $v\in E_{\ga(0)}$, then $T_{\,\ga}^A(v) = \widetilde{\ga}_A^{(v)}(1)$, where $\widetilde{\ga}_A^{(v)}$ is the unique $A$-horizontal lifting of $\ga$ satisfying $\widetilde{\ga}_A^{(v)}(0)=v$. We then have the following general result (which holds without assuming that $E$ is self-conjugate).

\begin{lemma}\label{parallel_transport_and_pullback_conn}
Denote by $\si(\ga)$ the path $\si\circ\ga:[0;1]\lra M$ and let $\si_E:E\lra \si(E)$ be the map defined in \eqref{pullback_map}. Then the parallel transport operator along the connection $\si(A)$ induced by $A$ on $\si(E)$ satisfies $T_{\si(\ga)}^{\si(A)} = \si_E T_{\,\ga}^A\si_E^{-1}$.
\end{lemma}

\begin{proof}
This follows directly from the definition of parallel transport in $\si(E)$. Indeed, by unicity of horizontal liftings, one has the following commutative diagram
$$
\xymatrix{
 & E \ar[rr]^{\si_{E}} \ar[d] & & \si(E) \ar[d] \\
[0;1] \ar[r]^{\ga} \ar@{-->}[ru]^{\widetilde{\ga}_{A}^{(v)}} \ar@{-->}[rrru]^{\widetilde{\si(\ga)}_{\si(A)}^{\si_E(v)}} & M \ar[rr]^{\si} & & M,
}
$$ which proves the lemma.
\end{proof}

\noindent Let us now fix a class $[c]\in H^2(\Si;\cZ(\GL(r;\C)))\simeq\{\pm 1\}$. Since $\Si\simeq \Z/2\Z$, we can just think of the cocycle $c$ itself as being $\pm 1$ (see \eqref{obs_on_H2}). We assume from now on that $E$ is either Real or Quaternionic, meaning that there is given an isomorphism $\phi_\si:\si(E)\lra E$ satisfying $\si(\phi_\si)=c\phi_\si^{-1}$ and we set $\tau=\phi_\si\si_E$, as in \eqref{real_or_quat_structure_on_self_conjugate_bundles}, so that $\tau^2=c\Id_E$. Recall from \eqref{the_involution} that in this situation we have an induced action of $\Si$ on the space of linear connections on $E$, defined by the involution $\beta(A) = (\phi_{\si}^{-1})^*\si(A)$.

\begin{proposition}\label{par_transport_inv_conn}
Let $(E,\tau)$ be a Real or Quaternionic vector bundle and let $A$ be a $\Si$-invariant connection on $E$. Then the parallel transport operators and $T_{\,\ga}^A$ and $T^A_{\si(\ga)}$ satisfy $T^A_{\si(\ga)} = \tau T_{\,\ga}^A\tau^{-1}$. Equivalently, we have a commutative diagram
$$
\begin{CD}
E_{\si(\ga(0))} @>{T^A_{\si(\ga)}}>> E_{\si(\ga(1))} \\ 
@A{\tau}AA @A{\tau}AA \\
E_{\ga(0)} @>{T^A_{\,\ga}}>> E_{\ga(1)}.
\end{CD}
$$
\end{proposition}

\begin{proof}
By Lemma \ref{parallel_transport_and_pullback_conn}, $T_{\,\ga}^A = \si_E^{-1}T_{\si(\ga)}^{\si(A)}\si_E = (\phi_\si\si_E)^{-1}(\phi_\si T_{\si(\ga)}^{\si(A)}\phi_\si^{-1})(\phi_\si\si_E)$, hence $T_{\,\ga}^A = \tau^{-1}T_{\si(\ga)}^{(\phi_\si^{-1})^*\si(A)}\tau= \tau^{-1} T_{\si(\ga)}^A\tau$, which proves the proposition. 
\end{proof}

\noindent Finally, let us choose a point $x\in M$ and a frame of $E_x$ (i.e.\ a $\C$-linear isomorphism $E_x\simeq \C^r$). The holonomy group of a connection $A$ at $x$ is the sub-group $\Holx$ of $\GL(r;\C)$ generated by all parallel transport operators $T_{\ga}^A$ along loops at $x$. We will now show that when $A$ is $\Si$-invariant, we can enlarge the holonomy group to obtain a $\Si$-augmentation $\HolxSi$ making the following diagram commute:

\begin{equation}\label{enlarged_hol_picture}
\begin{CD}
1 @>>> \Holx @>>> \HolxSi @>>> \Si @>>> 1 \\
@. @VVV @VVV \| @.\\
1 @>>> \GL(r;\C) @>>> \GL(r;\C)\times_c \Si @>>> \Si @>>> 1
\end{CD}
\end{equation}

\noindent To do that, we use the description of $\piRx$ in terms of paths in $M\times E\Si$ recalled in Section \ref{orbifold_groups}. Based on the definition of the set $\widetilde{P}_x$ in \eqref{paths_description}, we saw that the group $\piRx$ could be thought of as the set of homotopy classes of pairs $\eta=(\ga,\lambda)$ where $\ga:[0;1] \lra M$ and $\lambda\in\Si$ satisfy $\ga(0)=x$ and $\ga(1)=\lambda^{-1}(x)$, the composition $\eta_1\eta_2$ of two such pairs being defined as in \eqref{composition_of_paths}. The parallel transport operator is a $\C$-linear map $T_\ga^A:E_x\lra E_{\lambda^{-1}(x)}$
 and, by composing it by $\tau_\lambda$ (which we define to be either $\Id_E$ or $\tau$, depending on whether $\lambda=1_\Si$ or $\si$), we obtain a map $\tau_\lambda\circ T_\ga^A:E_x\lra E_x$, which is $\C$-linear or $\C$-anti-linear, depending on $\lambda\in\Si$. 
 
\begin{definition}\label{matrix_form}
Given $\eta=(\ga,\lambda)\in\widetilde{P}_x$, the matrix of the map $\tau_\lambda\circ T_\ga^A:E_x\lra E_x$ in the given frame of $E_x$ is the unique element $g_\ga\in\GL(r;\C)$ such that, for all $v\in\C^r\simeq E_x$,  $(\tau_\lambda\circ T_\ga^A)(v) = g_\ga \eps_\lambda(v)$, where $\eps_\lambda(v)= v$ if $\lambda=1_\Si$ and $\eps_\lambda(v)=\ov{v}$ if $\lambda=\si\in \Si$.
\end{definition}

\noindent Recall that $\Si$ acts on $\GL(r;\C)$ via $\si(g)=\ov{g}$. Then the map $\eps:\Si\lra \GL(r;\C)\rtimes\Si$ introduced in Definition \ref{matrix_form} is a group homomorphism and a section of the natural projection $\GL(r;\C)\rtimes\Si\lra\Si$. In particular, $\eps_\si g \eps_\si^{-1} = \si(g)=\ov{g}$. This will be useful in the proof of the next result.

\begin{theorem}\label{extended_holonomy}
Let $(E,\tau)$ be a Real or Quaternionic vector bundle and let $A$ be a Galois-invariant connection on $E$. Then there exists an extended holonomy group $$\HolxSi:=\{(g_\ga,\lambda) : \eta=(\ga,\lambda)\in\widetilde{P}_x\}$$ which, endowed with the projection homomorphism $(g_\ga,\lambda)\lmt\lambda$, becomes a sub-$\Si$-augmentation of $\GL(r;\C)\times_c\Si$ in the sense of Diagram \eqref{enlarged_hol_picture}, where $c=+1$ is $E$ is Real and $c=-1$ if $E$ is Quaternionic.\\ If $E$ is Hermitian and we choose $c$ to lie in $H^2(\Si;\cZ(\U(n)))\simeq\{\pm1\}$, then the extended holonomy group $\HolxSi$ is a sub-$\Si$-augmentation of $\U(r)\times_c\Si$.
\end{theorem}

\begin{proof}
It suffices to show that, for all $\eta_ 1=(\ga_1,\la_1)$, $\eta_2=(\ga_2,\la_2)$ in $\widetilde{P}_x$, the equality
\begin{equation}\label{product_in_hol_gp}
(g_{\ga_1\ga_2},\lambda_1\lambda_2) = (g_{\ga_1},\lambda_1)(g_{\ga_2},\lambda_2)
\end{equation} holds in $\GL(r;\C)\times_c\Si$. For the proof, it is convenient to think of $\Id_E$ and $\tau$ as a family of maps $(\tau_\la)_{\la\in\Si}$ such that, for all $(\la_1,\la_2)\in\Si\times\Si$, $\tau_{\la_1}\tau_{\la_2} = c(\la_1,\la_2)\tau_{\la_1\la_2}$, where $c$ is now seen as an abstract $2$-cocycle. By \eqref{composition_of_paths}, we have that $\eta_1\eta_2=(\lambda_2^{-1}(\ga_1)\ga_2,\lambda_1\lambda_2)$ so, by Definition \ref{matrix_form}, $g_{\ga_1\ga_2}\in \GL(r;\C)$ is the matrix such that for all $v\in\C^r\simeq E_x$, $$(\tau_{\lambda_1\lambda_2}T^A_{\lambda_2^{-1}(\ga_1)\ga_2})(v) = g_{\ga_1\ga_2}\eps_{\lambda_1\lambda_2}(v).$$ By the usual properties of parallel transport combined with Proposition \ref{par_transport_inv_conn}, we have $T^A_{\lambda_2^{-1}(\ga_1)\ga_2} = T^A_{\lambda_2^{-1}(\ga_1)} T^A_{\ga_2} = (\tau_{\lambda_2^{-1}}T^A_{\ga_1}\tau_{\lambda_2^{-1}}^{-1}) T^A_{\ga_2}$ and, since $\tau_{\lambda_2^{-1}} = c(\lambda_2,\lambda_2^{-1}) \tau_{\lambda_2}^{-1}$ and $c(\la_2,\la_2^{-1})\in \cZ(\GL(r;\C))$, we obtain $T^A_{\lambda_2^{-1}(\ga_1)\ga_2} = (\tau_{\la_2}^{-1} T^A_{\ga_1} \tau_{\la_2}) T^A_{\ga_2}$, hence 
\begin{equation}\label{towards_matrix_form}
\tau_{\la_1\la_2}T^A_{\la_2^{-1}(\ga_1)\ga_2} = c(\la_1,\la_2)^{-1} (\tau_{\la_1}T^A_{\ga_1})(\tau_{\la_2} T^A_{\ga_2}).
\end{equation} Moreover, for all $v\in\C^r\simeq E_x$, 
\begin{eqnarray*}
(\tau_{\la_1} T^A_{\ga_1}) (\tau_{\la_2} T^A_{\ga_2}) (v) = g_{\ga_1}\eps_{\la_1}\big(g_{\ga_2}\eps_{\la_2}(v)\big)
& = & g_{\ga_1} (\eps_{\la_1} g_{\ga_2} \eps_{\la_1}^{-1})\eps_{\la_1}\eps_{\la_2} (v) \\
& = & g_{\ga_1} \la_1(g_{\ga_2}) \eps_{\la_1\la_2} (v)
\end{eqnarray*} where, on the last line, $\la(g)=g$ if $\la=1_\Si$ and $\la(g)=\ov{g}$ if $\la=\si\in\Si$, as noted just after Definition \ref{matrix_form}. So, in matrix form, Equation \eqref{towards_matrix_form} becomes 
$$\big(g_{\ga_1\ga_2},\la_1\la_2\big) = \big(c(\la_1,\la_2)^{-1},1\big) \big(g_{\ga_1}\la_1(g_{\ga_2}),\la_1\la_2\big) = (g_{\ga_1},\la_1\big)\big(g_{\ga_2},\la_2\big),$$
where the last equality follows from the definition of the group law in $\GL(r;\C)\times_c\Si$, which was recalled in \eqref{gp_law_on_our_aug}. Thus, we have indeed proved Equality \eqref{product_in_hol_gp}.
\end{proof}

\subsection{Orbifold representations associated to invariant connections}

We now investigate the consequences of Theorem \ref{extended_holonomy} for the Narasimhan-Seshadri correspondence. We have seen in Theorem \ref{main_step_towards_NS} that  a poly-stable Real or Quaternionic structure on a Hermitian vector bundle $(E,\tau)$ of rank $r$ and degree $d$ is defined by a Galois-invariant, projectively flat unitary connection $A$ which is unique up to $\tau$-unitary gauge. By Theorem \ref{extended_holonomy}, the extended holonomy group of such a connection $A$ fits into the commutative diagram 
\begin{equation}\label{extended_hol_map_diagram}
\begin{CD}
1 @>>> \widetilde{L}_x @>>> \widetilde{P}_x @>>> \Si @>>> 1\\
@. @VVV @VV{\widetilde{\chi}}V \| @.\\
1 @>>> \Holx @>>> \HolxSi @>>> \Si @>>> 1
\end{CD}
\end{equation} where the first row is a short exact sequence of pointed sets and the extended holonomy group $\HolxSi$ is a sub-$\Si$-augmentation of $\U(r)\times_c\Si$ ($c$ being equal to $\pm 1$, depending on whether $\tau$ is Real or Quaternionic). In the present section, we wish to answer the following question: when does the pointed map $\widetilde{\chi}:\widetilde{P}_x \lra \U(r)\times_c\Si$ in Diagram \eqref{extended_hol_map_diagram} induce a group homomorphism 
\begin{equation}\label{top_trivial_case} \chi:(\piRx\simeq \widetilde{P}_x/\mathrm{homotopy}) \lra \U(r)\times_c\Si\ ?
\end{equation} As for the usual fundamental group $\piCx$, this can of course not happen unless $d=0$. Indeed, for $A$ projectively flat, the pointed map $\widetilde{\rho}:=\widetilde{\chi}|_{\widetilde{L}_x}:\widetilde{L}_x\lra \U(r)$ only induces a group homomorphism $\overline{\rho}:\piCx\lra \PU(r)$ in general. Equivalently, $\widetilde{\rho}$ induces a pointed map $\overline{\rho}:\piCx\lra\U(r)$ which is not a group homomorphism but satisfies \begin{equation}\label{almost_gp_hom_first}
\overline{\rho}(\ga_1\ga_2) = e^{i\frac{2\pi}{r}\alpha(\ga_1,\ga_2)}\overline{\rho}(\ga_1)\overline{\rho}(\ga_2),
\end{equation} where $\alpha:\piCx\times\piCx \lra \Z$ is a $2$-cocycle representing the cohomology class $d=c_1(E)\in H^2(M;\Z)$. As we saw in Section \ref{construction_of_bundles_section}, $d$ is also the class of the (central) extension $\piL$ of $\piCx$ by $\Z$, where $S(L)$ is the Seifert manifold associated to an arbitrary smooth line bundle $L$ of degree $d$. Then, since $1\lra S^1 \lra \U(r) \lra \PU(r) \lra 1$ is a central extension, a pointed map $\overline{\rho}$ satisfying \eqref{almost_gp_hom_first} induces a group homomorphism $\rho:\piL\lra\U(r)$, satisfying, for all $n\in\Z=\cZ(\piL)$, $\rho(n) =\exp(i\frac{2\pi}{r}n)I_r$. In the Real or Quaternionic case, the situation is very similar. The pointed map $\widetilde{\chi}: \widetilde{P}_x\lra \U(r)\times_c\Si$ induces a group homomorphism $\overline{\chi}: \piRx \lra \PU(r)\rtimes\Si$, where $\Si$ acts on $\PU(r)$ by complex conjugation, or equivalently a pointed map $\overline{\chi}:\piRx\lra  \U(r)\times_c\Si$, compatible with the projections to $\Si$ and satisfying 
\begin{equation}\label{almost_gp_hom_secalmost_gp_second}
\overline{\chi}(\eta_1\eta_2) = e^{i\frac{2\pi}{r}\at(\eta_1,\eta_2)} \overline{\chi}(\eta_1)\overline{\chi}(\eta_2),
\end{equation} where $\at:\piRx\times\piRx\lra \Z$ is any $2$-cocycle extending the previous cocycle $\alpha$. Since $\alpha$ was representing a smooth line bundle of degree $d$ and $\at$ represents a smooth Real line bundle $(L,\tauL)$, saying that $\at$ extends $\alpha$ is equivalent to saying that $L$ has degree $d$. By Theorem \ref{top_classif}, the isomorphism class of such an $(L,\tauL)$ is unique when $M^\si=\emptyset$ but, when $M^\si$ has, say, $n>0$ connected components, then there are $2^{n-1}$ possibilities for the isomorphism class of $(L,\tauL)$. This constitutes a difference with the complex case, where the isomorphism class of the line bundle $L$ depended only on $d$. But this is not an issue because, by Remark \ref{ext_class_vs_aug_class}, any two such choices will give rise to isomorphic $\Si$-augmentations $\piLR$. Once chosen a $\tauL$, we obtain, as in Section \ref{construction_of_bundles_section}, a Real Seifert manifold $(S(L),\tauL)$ and an associated (non-central) extension $\piLR$ of $\piRx$ by $\Z$. Then, since $1\lra S^1\lra  \U(r)\times_c \Si \lra \PU(r)\rtimes\Si\lra 1$ is a central extension, a pointed map $\overline{\chi}:\piRx\lra \U(r)\times_c\Si$ compatible with the projections to $\Si$ and satisfying Condition \eqref{almost_gp_hom_secalmost_gp_second} induces a homomorphism of $\Si$-augmentations $\chi:\piLR\lra \U(r)\times_c\Si$, satisfying, for all $n\in\Z$, $\chi(n) = \exp(i\frac{2\pi}{r}n)I_r$, as in Diagram \eqref{morphisms_of_extensions} and Equation \eqref{cond_on_morphisms}. By Remark \ref{generalized_orbifold_rep_space}, when $d=0$, we can think of such a $\chi$ as a homomorphism of $\Si$-augmentations from $\piRx$ to $\U(r)\times_c\Si$, as expected in \eqref{top_trivial_case}. We have therefore proved the following result.
\begin{theorem}\label{holonomy_rep}
Let $(E,\tau)$ be a smooth Real or Quaternionic vector bundle of rank $r$ and degree $d$. Let $c=\pm 1$ be defined by the equation $\tau^2=c\Id_E$. Let $(L,\tauL)$ be any smooth Real line bundle of degree $d$ and let $(S(L),\tauL)$ be the associated Real Seifert manifold. Let $A$ be a Galois-invariant, projectively flat unitary connection on $E$. Then, taking the extended holonomy of the connection $A$ in the sense of Theorem \ref{extended_holonomy} induces a homomorphism of $\Si$-augmentations $\chi_A:\piLR\lra \U(r)\times_c\Si$, whose $\U(r)$-conjugacy class only depends on the $\tau$-unitary gauge orbit of $A$.
\end{theorem}
\noindent In view of \eqref{GIT_pic_Real_quat_case}, Theorem \ref{holonomy_rep} means that we have constructed a holonomy map 
\begin{equation}\label{from_bundles_to_reps}
\mathrm{Hol}: \begin{array}{ccc} \cM^{ss}(E,\tau) & \lra & \cR_c(r,d)\\
 A & \lmt & \chi_A
\end{array}
\end{equation} where $c=\pm 1$ depending on whether $\tau$ is Real or Quaternionic and $\cR_c(r,d)$ is the representation space of $\piLR$ introduced in \eqref{our_rep_space}. We will see shortly that the collection of such holonomy maps for all possible topological types of Real or Quaternionic structures $\tau$ on $E$ provides an inverse to the maps in \eqref{map_from_reps_to_bundles}. But before we do that, let us study the special case where $M^\si\neq\emptyset$. As we saw in Proposition \ref{equivariant_rep}, if we choose $x\in M^\si$, then orbifold representations may be replaced by $\Si$-equivariant unitary representations of the usual fundamental group $\piC$, or more generally $\piL$. And we verified in \eqref{from_equiv_rep_to_bdles} that it was indeed easy, starting from a $\Si$-equivariant representation, to define a Real or Quaternionic structure on the associated projectively flat bundle. Conversely, we can check here that the (non-extended) holonomy representation $\rho_A:\piL \lra \U(r)$ of a Galois-invariant unitary connection $A$ on $E$ is indeed $\Si$-equivariant (for the action of $\Si$ on $\U(r)$ defined by $\si(u)=\ov{u}$ in the Real case and the action of $\Si$ on $\U(2r')$ defined by $\si(u)=J\ov{u}J^{-1}$ in the Quaternionic case). 
\begin{proposition}\label{equivariant_rep_from_invariant_conn}
Let $x\in M^\si$. Then, given a projectively flat connection $A$ on $E$, one has $\rho_{\beta(A)}=\si \rho_A \si^{-1}$. In particular, if $\beta(A)=A$, then $\rho_A$ is $\Si$-equivariant.
\end{proposition}
\begin{proof} 
The result is proved as in Proposition \ref{par_transport_inv_conn}. To translate it into matrix form, the frame $E_x\simeq \C^r$ must send the Real (resp.\ Quaternionic) structure  $\tau|_{E_x}$ to the canonical Real (resp.\ Quaternionic) structure $v\lmt\ov{v}$ (resp.\ $v\lmt J\ov{v}$) of $\C^r$ (resp.\ $\C^{2r'}$). The existence of such frames can be proved by induction on $r$ (resp.\ $r'$).
\end{proof}

\noindent In particular, it is not necessary, when $M^\si\neq\emptyset$, to use Theorem \ref{extended_holonomy} in order to define the holonomy map \eqref{from_bundles_to_reps}. It suffices to compose the map $A\lmt\rho_A$ from Proposition \ref{equivariant_rep_from_invariant_conn} with the isomorphism in Proposition \ref{equivariant_rep}:
\begin{equation*}
\mathrm{Hol}: \begin{array}{ccc}
\cM^{ss}(E,\tau) & \lra & \Hom^\Z(\piL;\U(r))^\Si/\U(r)^\Si \simeq \cR_c(r,d) \\
A & \lmt & \rho_A
\end{array}
\end{equation*} where $d=\deg(L)$ and $c=\pm 1$ depending on whether $\tau$ is Real or Quaternionic.

\subsection{Connected components of the representation variety}\label{CC}

The moduli space $\cM^{ss}(E,\tau)$ constructed in \eqref{GIT_pic_Real_quat_case} depends only on the topological type of $(E,\tau)$ and, by Theorem \ref{top_classif}, the latter is entirely determined by the numerical quantities $c=\tau^2$, $r=\rk(E)$, $d=\deg(E)$ and (when $c=+1$ and $M^\si\neq\emptyset$) $w=w_1(E^\tau)$. Let us then set 
\begin{equation}\label{future_CC}
\cM_c(r,d,\vw) := \left\{ 
\begin{array}{ll}
\cM_\R(r,d) & \mathrm{if}\ c=+1\ \mathrm{and}\ M^\si=\emptyset,\\
\cM_\R(r,d,\vw) & \mathrm{if}\ c=+1\ \mathrm{and}\ M^\si\neq\emptyset,\\
\cM_\H(r,d) & \mathrm{if}\ c=-1,
\end{array}
\right.
\end{equation} where the moduli spaces on the right-hand side are the ones introduced in Definition \ref{def_moduli_spaces_Real_Quat_bdles}. By the results of \cite{BHH} and \cite{Sch_JSG}, the sets $\cM_c(r,d,\vw)$ are non-empty connected topological spaces which can be embedded into $\ModCd^\Si$, possibly with some non-trivial intersection unless we restrict to the stable locus of $\ModCd$. In \eqref{from_bundles_to_reps}, we defined a map $\Hol:\cM_c(r,d,\vw) \lra \cR_c(r,d)$ and, in order to complete the proof of Theorem \ref{NS_over_R}, we now want to specify the image of that map (see Theorem \ref{NS_over_R_fixed_top_type}). To that end, let us introduce new sets $\cR_c(r,d,\vw)$. When $c=-1$  or $M^\si=\emptyset$, $\cR_c(r,d,\vw)$ will coincide with $\cR_c(r,d)$ by definition. But when $c=+1$ and $M^\si\neq\emptyset$, $\cR_c(r,d,\vw)$ will be the subset of $\cR_c(r,d)$ defined as follows. Let us denote by $\ga_1,\ldots,\ga_n$ the connected components of $M^\si$. We view $\ga_k$ as a loop in $M$ by picking a base point $x_k\in \ga_k$. Recall from Proposition \ref{equivariant_rep} that, if we choose a base point $x\in M^\si$, the representation variety $\cR_c(r,d)$ can be seen as the set of equivalence classes of $\Si$-equivariant representations $\rho:\piL\lra\U(r)$ satisfying, for all $n\in\Z\subset\piL$, $\rho(n)=\exp(i\frac{2\pi}{r}n)I_r$. In particular $\det\rho(n)=1$ for all $n\in\Z$. Let us now pick a path $\delta_k$ from $x$ to $x_k$. Then we have a loop $\eta_k:=\delta_k^{-1} \ga_k \delta_k$, based at $x$. As in Section \ref{case_with_real_points}, the Real Seifert fibration $S^1\lra S(L) \lra M$ induces a short exact sequence of $\Si$-equivariant group homomorphisms $$0\lra\Z\lra\pi_1(S(L);\ov{x})\lra\piCx\lra 1,$$ where $\ov{x}\in\Fix(\tauL)$ lies in the fiber of $S(L)$ above $x$ and $\Si$ acts on $\Z$ by $n\lmt (-n)$. So we can lift $\eta_k$ to an element $\alpha_k\in\piL$. Morever, $\si(\eta_k) = \zeta_k^{-1}\eta_k\zeta_k$, where $\zeta_k:=\si(\delta_k^{-1})\delta_k$ is a loop at $x$, so we can also lift $\zeta_k$ to some $\beta_k\in\piL$ and we see that $\si(\alpha_k)$ and $\beta_k^{-1}\alpha_k\beta_k$ differ only by some $n_k\in\Z$. In particular, $\det\rho(\alpha_k)=\det\rho(\si(\alpha_k))=\si(\det\rho(\alpha_k))\in\U(1)^\Si=\O(1)$ and is independent of the various choices of liftings that we have made. This implies the existence of a map
\begin{equation}\label{real_obstruction_map}
\cW: 
\begin{array}{ccc}
\cR_c(r,d) & \lra & \O(1)^n\\
\rho & \lmt & \big(\det\rho(\alpha_1),\ldots,\det\rho(\alpha_n)\big)
\end{array}
\end{equation} that, in analogy with \cite{Goldman_top_comp}, we may call the Real obstruction map.

\begin{definition}\label{sub_repvar}
Assume that $c=+1$ and $M^\si\neq\emptyset$. Given $w=(s_1,\ldots,s_n)\in\O(1)^n$, we set $$\cR_c(r,d,\vw):=\cW^{-1}(s_1,\ldots,s_n)$$ where $\cW$ is the Real obstruction map \eqref{real_obstruction_map}. As we will see below, this set is non-empty if and only if $s_1\ldots s_n=(-1)^{d\,\mod\,2}$. If $c=-1$ or $M^\si=\emptyset$, we set $\cR_c(r,d,\vw):=\cR_c(r,d)$. In all cases we have: $$\cR_c(r,d)=\bigsqcup_{w} \cR_c(r,d,\vw).$$
\end{definition}

\noindent If $c=+1$ and $M^\si\neq\emptyset$, then the Real vector bundle $(\cE,\tau)$ associated to $\rho$ by means of Theorem \ref{constr_of_bdles_prop} satisfies $w_1(\cE^\tau)=w$, by definition of the first Stiefel-Whitney class. Conversely, if $A$ is a $\Si$-invariant, projectively flat unitary connection on a Real Hermitian vector bundle $(E,\tau)$ of topological type $(r,d,\vw)$, then it follows from Theorem \ref{holonomy_rep} that taking the holonomy of $A$ defines a representation $\rho\in\cR_c(r,d,\vw)$. We have thus proved the following result (the Narasimhan-Seshadri correspondence for Real and Quaternionic vector bundles of fixed topological type).

\begin{theorem}\label{NS_over_R_fixed_top_type}
For any $c=\pm 1$ and any topological type $(r,d,\vw)$ of Real or Quaternionic vector bundle, the holonomy map sets up a homeomorphism $$\mathrm{Hol}:\cM_c(r,d,\vw) \overset{\simeq}{\lra} \cR_c(r,d,\vw)$$ where $\cM_c(r,d,\vw)$ is the moduli space defined in \eqref{future_CC} and $\cR_c(r,d,\vw)$ is the representation space introduced in Definition \ref{sub_repvar}.
\end{theorem}

\noindent Using the results of \cite{BHH} and \cite{Sch_JSG}, we then have the following corollary.

\begin{corollary}\label{CC_of_rep_var}
Assume that $c=+1$ and $M^\si$ has $n>0$ connected components, then the connected components of the representation space $\cR_c(r,d)$ introduced in Definition \ref{orbifold_rep_of_extensions} are the $2^{n-1}$ subset $\cR(r,d,\vw)$, where $w=(s_1,\ldots,s_n)\in (\Z/2\Z)^n\simeq\O(1)^n$ satisfies $s_1+\ldots+s_n=d\,\mod\,2$.
\end{corollary}

\noindent Corollary \ref{CC_of_rep_var} fits well in the theory of representation varieties of fundamental groups, where, following the path set by Goldman in \cite{Goldman_top_comp}, connected components of representation spaces are often distinguished by the topological invariants of the bundles associated to those representations, notably the characteristic classes of such bundles. We also note that, although the connected components of $\cR_c(r,d)$ are known for all $c$, $r$ and $d$, the same is not true in general for the $\Si$-invariant part $\cR(r,d)^\Si$ of the usual representation space $\cR(r,d)$, whose definition was recalled in Remark \ref{usual_rep_var}, unless $r$ and $d$ are coprime or one restricts to the locus consisting of irreducible representations, for which we can use the vector bundle version of the statement, proved in \cite{Sch_JSG}.


\end{document}